\documentclass{amsart}

\usepackage{graphicx} 
\usepackage[utf8]{inputenc}
\usepackage{amsmath}
\usepackage{amssymb}
\usepackage{amsfonts}
\usepackage{amsthm}
\usepackage{xcolor}
\usepackage{geometry}
\usepackage{hyperref}
\usepackage{array}
\usepackage{tikz-cd}
\usepackage{afterpage}

\usepackage{adjustbox}
\newcolumntype{L}{>{$}l<{$}} 
\newcolumntype{C}{>{$}c<{$}} 

\theoremstyle{plain}
\newtheorem*{thm*}{Theorem}
\newtheorem{thm}{Theorem}[section]
\newtheorem*{lemma*}{Lemma}
\newtheorem{lemma}[thm]{Lemma}
\newtheorem{pro}[thm]{Proposition}
\newtheorem*{pro*}{Proposition}
\newtheorem{cor}[thm]{Corollary}
\newtheorem*{cor*}{Corollary}
\newtheorem*{con*}{Conjecture}
\newtheorem{con}[thm]{Conjecture}

\theoremstyle{definition}
\newtheorem{df}[thm]{Definition}
\newtheorem*{df*}{Definition}

\theoremstyle{remark}
\newtheorem{rem}[thm]{Remark}
\newtheorem*{rem*}{Remark}
\newtheorem{ex}[thm]{Example}
\newtheorem*{ex*}{Example}

\newcommand{\multi}{{\underline{\eta}}} 
\newcommand{\mzeta}{{\underline{\zeta}}} 
\newcommand{\cvar}{{\underline{c}}} 
\newcommand{\svar}{{\underline{s}}} 
\newcommand{\tvar}{{\underline{t}}} 
\newcommand{\avar}{{\underline{a}}} 
\newcommand\N{\widetilde{N}} 
\newcommand\M{\widetilde{M}} 
\newcommand\F{\operatorname{F}} 
\newcommand\FF{\widetilde{\operatorname{F}}} 
\newcommand{\sq}{\,\dot{\sqcup}\,} 

\newcommand{\ThS}[1]{{\operatorname{Th}^S_{#1}}} 
\newcommand{\ThT}[1]{{\operatorname{Th}^T_{#1}}} 
\newcommand{\MS}{\exp(\MSlog)} 
\newcommand{\MSlog}{S_\varnothing} 
\newcommand{\ThScl}[1]{{\operatorname{Th}}^{S,cl}_{#1}} 
\newcommand{\ThTcl}[1]{{\operatorname{Th}}^{T,cl}_{#1}} 
\newcommand{\KazT}[1]{S_{#1}} 
\newcommand{\KazS}[1]{R_{#1}} 
\newcommand{\ThTMa}[2]{\operatorname{Th}^T_{#1,\le #2}} 
\newcommand{\ThSMa}[2]{\operatorname{Th}^S_{#1,\le #2}} 
\newcommand{\Slocus}[2]{\Sigma^{S}_{#1}(#2)} 
\newcommand{\Tlocus}[2]{\Sigma^{T}_{#1}(#2)} 

\DeclareMathOperator{\E}{\mathcal E}
\DeclareMathOperator{\K}{\mathcal K}
\DeclareMathOperator{\OO}{\mathcal O}
\DeclareMathOperator{\Ideal}{\mathfrak m}
\DeclareMathOperator{\codim}{codim}
\DeclareMathOperator{\scodim}{scodim}
\DeclareMathOperator{\tcodim}{tcodim}
\newcommand\III{I\!I\!I}
\newcommand\cl{cl}

\newcommand{\CC}{\mathcal{C}}
\newcommand{\CCv}{\tilde{\mathcal{C}}}
\newcommand{\Ma}{M\!a}
\newcommand\rhoS{\rho^S}
\newcommand\rhoT{\rho^T}

\newcommand{\xto}[1]{\stackrel{#1}{\longrightarrow}}
\newcommand{\Aut}{\operatorname{Aut}}
\newcommand{\csm}{\operatorname{csm}}
\newcommand{\ssm}{\operatorname{ssm}}
\newcommand{\coh}{\operatorname{H}}
\newcommand{\id}{\operatorname{id}}
\newcommand{\eu}{\operatorname{eu}}
\newcommand{\pt}{\operatorname{pt}}
\newcommand{\im}{\operatorname{im}}

\newcommand{\A}{\mathcal{A}}
\newcommand{\ZZ}{\mathbb{Z}}
\newcommand{\NN}{\mathbb{N}}
\newcommand{\QQ}{\mathbb{Q}}
\newcommand{\C}{\mathbb{C}}
\newcommand{\TT}{\mathbb{T}}
\newcommand{\CP}{{\C}\!P}

\title{Higher characteristic classes of multisingularity loci}

\author{Jakub Koncki}
\address{Institute of Mathematics, University of Warsaw, Warsaw, Poland}

\author{Rich\'ard Rim\'anyi}
\address{Department of Mathematics, UNC Chapel Hill, Chapel Hill, NC, USA}

\date{Spring 2025}

\begin{document}

\begin{abstract}
A map between manifolds induces stratifications of both the source and the target according to the occurring multisingularities. In this paper, we study universal expressions—called higher Thom polynomials—that describe the Segre–Schwartz–MacPherson class of such multisingularity loci. We prove a Structure Theorem reducing these Thom polynomials to the data of a linear series associated with each multisingularity. The series corresponding to the empty multisingularity, referred to as the Master Series, plays a distinguished role. Motivated by connections with geometric representation theory, we further prove an Interpolation Theorem that allows Thom polynomials to be computed algorithmically within Mather’s range of nice dimensions. As an application, we derive an explicit formula for the image Milnor number of quasihomogeneous germs, providing one side of the celebrated Mond conjecture, computable up to the theoretical bound.
\end{abstract}

\maketitle

\section{Introduction}

\subsection{Multisingularity loci}
A complex algebraic map $f: M \to N$ between complex manifolds induces stratifications of both $M$ and $N$ by {\em multisingularities}. A multisingularity $\multi=\{\eta_1,\ldots,\eta_r\}$ is a multiset of singularity types, and the corresponding {\em target locus}
$
\Tlocus{\multi}{f} \subset N
$
consists of points having exactly $r$ preimages where $f$ has singularities of types $\eta_1, \ldots, \eta_r$. The nonsingular behavior, denoted $A_0$, is regarded as a singularity type; for example, $\Tlocus{{A_0, A_0}}{f}$ is the locus of ordinary double points. With the notations $A_2 = \C[t]/(t^3)$ and $I_{23} = \C[x,y]/(xy, x^2 + y^3)$, the locus $\Tlocus{{A_0, A_2, I_{23}}}{f}$ consists of points of $N$ having three preimages: one nonsingular and two with the indicated local algebras. The algebraic encoding of singular behavior will be recalled in Section~\ref{s:sing}.

The superscript $T$ indicates that $\Tlocus{\multi}{f}$ lies in the {\em target}; the corresponding {\em source locus} $\Slocus{\multi}{f} \subset M$ is defined analogously for multisingularities with one distinguished component.

\subsection{Universal polynomials: the Thom principle}
A central aim of global singularity theory is to express characteristic classes of the loci $\Tlocus{\multi}{f}$ and $\Slocus{\multi}{f}$ in terms of characteristic classes associated with the map $f$, defined by
\[
1 + c_1(f) + c_2(f) + \cdots = \frac{f^*(c(TN))}{c(TM)},
\qquad
s_\lambda(f) = f_*\left( \prod_i c_{\lambda_i} \right).
\]
The guiding {\em Thom principle} (see Section~\ref{sec:ThomPrinciple}) --- a theorem, conjecture, or heuristic depending on context --- asserts that for any characteristic class theory $cl$ applicable to subvarieties, there exist universal polynomials $P_\multi$ and $Q_\multi$, depending only on the multisingularity $\multi$, such that for suitable maps $f$,
\[
cl\left(\Slocus{\multi}{f} \subset M\right)
= P_\multi \left(c_i(f), f^*(s_\lambda(f))\right),
\qquad
cl\left(\Tlocus{\multi}{f} \subset N\right)
= Q_\multi \left(s_\lambda(f)\right).
\]
The strength of this principle lies in the {\em universality} of $P_\multi$ and $Q_\multi$: they are independent of the specific map $f$, depending only on $\multi$. Since the characteristic classes $c_i(f)$ and $s_\lambda(f)$ are often readily computable (e.g. homotopy invariants of $f$), the Thom principle provides a powerful bridge from these accessible data to the intricate geometry of multisingularity loci.

A classical example of the Thom principle is the (target) {\em double point formula}
\[
[\overline{\Tlocus{{A_0, A_0}}{f}}] = (s_0^2 - s_l)(f) \in H^*(N),
\]
valid for sufficiently nice maps $f: M^m \to N^{m+l}$.
Here the fundamental cohomology class of the closure of the double point locus is expressed by the {\em target Thom polynomial} $s_0^2 - s_l$.

\subsection{Segre–Schwartz–MacPherson class}

In this paper we consider a refinement of the notion of fundamental class: the {\em Segre–Schwartz–MacPherson} (or $\ssm$) class
\[
\ssm(\Sigma \subset N) = [\overline{\Sigma}] + \text{higher order terms} \quad \in H^*(N),
\]
associated with a (possibly singular or non-closed) subvariety $\Sigma \subset N$ of a smooth ambient space $N$.

The concept of $\ssm$ classes has two origins.
First, up to normalization and suitable identifications, they can be viewed as motivic analogues of the total Chern class of the tangent bundle, extending this notion to singular varieties. This perspective traces back to the foundational works of Deligne, Grothendieck, MacPherson, and Schwartz; see the surveys in Chapters 5-7 of \cite{BrasseletBook}.
Second, they are closely related to the stable envelope characteristic classes of Maulik and Okounkov \cite{MO}, central to geometric representation theory; the connection between stable envelopes and $\ssm$
classes has been developed in \cite{RV, FRcsm, AMSS_csm}.

To illustrate the additional information carried by $\ssm(\Sigma)$ beyond the fundamental class, note that when $N$ is a projective space, the class $[\overline{\Sigma}]$ determines the degree of $\overline{\Sigma}$, which may be interpreted as the Euler characteristic of a general linear section of $\Sigma$ of complementary dimension. In contrast, $\ssm(\Sigma)$ encodes the Euler characteristics of general linear sections of $\Sigma$ of all dimensions \cite{ohmotoCHI, aluffi}.

The Thom principle for $\ssm$-classes of monosingularities --- defining the first SSM-Thom polynomials --- was introduced in \cite[Ch.~5]{OhmotoCamb}; see also \cite{ohmotoSMTP}. Real, mod 2, analogues of $\ssm$ classes in singularity theory are explored in \cite{MatszFeher}.

\subsection{SSM–Thom polynomials for multisingularities}

In what follows, we establish (in certain cases) and conjecture (in others) the existence of {\em SSM–Thom polynomials} of multisingularities, denoted
\[
\ThS{\multi}, \quad \ThT{\multi}.
\]
These are universal power series depending only on the multisingularity $\multi$, not on the map $f$, and they express the $\ssm$ classes of the source and target multisingularity loci in terms of the characteristic classes of~$f$. 

Our first main result is a {\em Structure Theorem} for $\ThS{\multi}$ and $\ThT{\multi}$: Corollaries~\ref{cor:indT}, \ref{cor:IndTver}. 
The key observation is that the typically intricate high-degree expressions for $\ThT{\multi}$ can be encoded compactly. Specifically, for every multisingularity $\multi$ there exists a formal power series
$
\KazT{\multi} \in \QQ[[s_\lambda]],
$
linear in the variables $s_\lambda$, such that
\begin{equation}\label{eq:IntoMain}
\boxed{
\sum_{\multi}
\frac{\ThT{\multi}}{|\Aut(\multi)|}  t^\multi
=
\exp\bigg(
\sum_{\multi}
\frac{\KazT{\multi}}{|\Aut(\multi)|}  t^\multi
\bigg)
\in \QQ[[\svar, \tvar]],
}
\end{equation}
(and see Corollary~\ref{cor:IndS} for the source version). 
Here $t^\multi$ denotes monomials in formal variables indexed by monosingularities; 
and the denominators are explicit factorials reflecting automorphism symmetries of $\multi$. This structure of SSM-Thom polynomials for multisingularities was conjectured, based on a comment by M.~Kazarian, in \cite[Sect.~6.4]{ohmotoSMTP}.

Equation~\eqref{eq:IntoMain} can also be rewritten as a {\em recursive} expression for the Thom polynomials $\ThT{\multi}$ in terms of the fundamental building blocks $\KazT{\multi}$. This structure --- whether written in exponential or recursive form --- is analogous to that established by Kazarian~\cite{KazaMulti} and proved by Ohmoto~\cite{TOmult} for the classical {\em fundamental class} Thom polynomials. However, even at the level of formulation, a crucial difference arises in the role of the {empty multisingularity} $\multi = \varnothing$. After all,
{\em to understand something, the first step is to understand nothing}.

\subsection{The Master Series \texorpdfstring{$\KazT{\varnothing}$}{KazT(varnothing)}}
For the classical (fundamental class) Thom polynomials one has $\KazT{\varnothing} = 0$, for trivial reasons. 
In contrast, for the SSM-Thom polynomials, the non-trivial series $\KazT{\varnothing}$ plays a central role; we therefore call it the {\em Master Series}.
As follows from~\eqref{eq:IntoMain}, the Master Series is a fundamental ingredient in the construction of every $\ThT{\multi}$. Or, 
to phrase it more philosophically: {\em emptiness is not void --- it is the source of all things, the foundation of existence.}

For every positive integer $l \ge 1$ there is a Master Series corresponding to the empty multisingularity for maps $M^* \to N^{*+l}$. For $l=1$, it is 
\[
\KazT{\varnothing}
= -s_{} + \tfrac{1}{2}s_{1} + \tfrac{1}{6}(7s_{2} - 2s_{11}) 
+ 
\tfrac{1}{4}( s_3 -5 s_{21} +s_{111} )
+
\cdots,
\]
see Example~\ref{ex:MS} and~\cite{TPP} for further terms and other $l$. Interestingly, $\KazT{\varnothing}$ involves nontrivial denominators which, for $l=1$, appear to coincide with the denominators of the Cauchy numbers of the second kind: $1,2,6,4,30,12,84,24,\ldots$.
By contrast, all computed series $\KazT{\multi}$ for nonempty $\multi$ have integer coefficients; see e.g. Figure~\ref{fig:Sl1}.

\subsection{Computations}
The theory of global singularities is inherently {\em computational}: explicit Thom polynomials often translate into concrete results in enumerative and algebraic geometry, as well as obstruction theory. In Section~\ref{s:interpolation} we describe a method to compute the SSM-Thom polynomials in broad generality. Let us outline its main ingredients.

Our starting point is the {\em interpolation method}, introduced in~\cite{rrtp}.
Although alternative approaches exist --- based on partial resolutions, iterated residues, or nonreductive quotients~\cite{BSz, FRannals, BercziNew, BSzmult},  --- interpolation remains among the most effective ones for computing classical (fundamental class) Thom polynomials.
However, in its original form it does not extend directly to the $\ssm$ setting.

An extension of this method was proposed by Ohmoto, and Ohmoto-Nekarda~\cite{ohmotoSMTP,nekarda1,nekarda2}, who reduced the computation of SSM-Thom polynomials to that of $\ssm$ classes of certain singular affine varieties. While the latter computations are often difficult in themselves, the method proved efficient for some $\multi$ in low degree. 

A further key input comes from {\em geometric representation theory}, specifically the {\em Maulik–Okounkov stable envelopes}. These are characteristic classes arising in quantum integrable systems, defined axiomatically. It is known that whenever both stable envelopes and $\ssm$ classes are defined, they are closely related.
This suggests that suitable variants of the MO-axioms should characterize SSM-Thom polynomials --- a fact established as part of our {\em Interpolation Theorem \ref{thm:Interpolation}}.

One of the MO-axioms, the {\em support axiom}, is of geometric nature. Remarkably, it becomes purely {\em algebraic} in the SSM-Thom setting, manifesting as an interpolation constraint. This is another fact that is part of our Interpolation Theorem.
However, to emphasize the delicacy of the situation, let us note that we do not know of any counterpart of this phenomenon for stable envelopes.

Combining these ideas, our computation of SSM-Thom polynomials for multisingularities (including the Master Series) becomes entirely algorithmic, with no geometric input. The only ingredient needed is the list of monosingularities (55 of them) in the Mather range together with their symmetries. 
The resulting formulas are implemented computationally; a selection of data is available at~\cite{TPP}.
For instance,
\[
\KazT{A_0A_0}^{l=2}
=
-s_2
+ (2s_{21} + 2s_3)
- (7s_4 + 7s_{31} + 3s_{211})
+ (36s_5 + 37s_{41} + 12s_{32} + 15s_{311} + 4s_{2111})
- \cdots.
\]

\subsection{Mond conjecture} 
We expect that the mentioned Structure Theorem and the Interpolation Theorem will have applications across several areas of geometry. In this paper, we present one such application: a contribution to Mond’s conjecture.

There are important theorems all around mathematics that link the interior features of a function (such as oscillatory behavior) to its exterior geometric or deformation-theoretic properties --- that is, how it sits inside an ambient space of functions. Classic examples include the Sturm–Liouville theory and the inequality “Milnor number $\geq$ Tjurina number” for complex analytic functions. Mond’s conjecture is an open problem of this same nature.

Let $f$ be finite germ $f : (\C^m, 0) \to (\C^{m+1}, 0)$ (see Section~\ref{s:Mond}). In the space of germs, consider the subset consisting of those that become equivalent to $f$ after reparametrizations of the source and target. The codimension of this subset is denoted by 
$\A_e\text{-codim}(f)$.

When $m\leq 14$, the germ $f$ has a stable perturbation, and its image is homotopy equivalent to a bouquet of spheres. The number of spheres in this bouquet, denoted $\mu_I(f)$, is called the {\em image Milnor number} of $f$. Mond’s conjecture~\cite[Rem.~8.1]{MNB} asserts that
$
\A_e\text{-codim}(f) \le \mu_I(f),
$
and that equality holds precisely when $f$ is quasihomogeneous.

This conjecture remains open, although its analogue for function germs is known and forms a cornerstone of the theory of singularities of functions. At the current stage, even concrete examples supporting the conjecture are of interest --- and this is precisely where Thom polynomials can make a contribution. Building on a result of Ohmoto~\cite{ohmotoSMTP}, see also \cite{PallaresPenafort}, the information encoded in our Master Series translates directly into a formula for the image Milnor number $\mu_I(f)$ of a quasihomogeneous map germ. We present this result in Section~\ref{s:Mond}; it provides an explicit expression for one side of Mond’s conjecture, valid up to the theoretical bound $M(1) = 14$.

\begin{rem}
As already noted, the appearance of denominators in the Master Series --- and consequently in Thom polynomials --- remains somewhat mysterious. It appears that the theory of multisingularities may reveal new number-theoretic constraints on certain combinations of characteristic classes of maps. We do not pursue this direction in the present paper; however, we illustrate the nature of possible results in Example~\ref{ex:ImageMilnorExample}.
\end{rem}

\subsection{Conventions}
Throughout the paper we work with even degree cohomology with rational coefficients, and will use the notation $\coh^*(X)=H^{2*}(X;\QQ)$. For the Euler class of $\xi$ we will write $\eu(\xi)$. In most of the paper we study maps from $m$ dimensions to $m+l$ dimensions with $l\geq 1$. We call $l$ the relative dimension of such maps. At some places, but not in the main theorems, we permit $l=0$. The upper indices $S$, $T$ always refer to source and target, that is, domain and codomain.

\subsection{Acknowledgments}
The first author was supported by National Science Centre (Poland) grant SO\-NA\-TI\-NA
2023/48/C/ST1/00002. The second author was supported by the U.S. National Science Foundation under Grant No. 2152309. Any opinions, findings, and conclusions or recommendations expressed in this material are those of the author(s) and do not necessarily reflect the views of the National Science Foundation.
\\
We are grateful to L. Feh\'er and T. Ohmoto for useful discussions on the topic.

\section{Singularities} \label{s:sing}
In this section we briefly recall some basic notions of singularity theory, for more details see \cite{MNB}.

\subsection{Contact monosingularities}
For $m,l\geq 0$ let $\E(m,m+l)$ be the vector space of germs of holomorhic maps $(\C^m,0)\to (\C^{m+l},0)$. The group of complex holomorphic diffeomorphism germs of $(\C^m\times \C^{m+l},0)$ of the form 
\[
\Phi(x,y)=(\phi(x), \psi(x,y))
\]
where $\psi(x,0)=0$, is called the contact group $\K(m,m+l)$. The group $\K(m,m+l)$ acts on $\E(m,m+l)$ via its action on the graph. Orbits of the action are called {\em contact monosingularities}, or simply {\em singularities}.

The local algebra of a germ $f:(x_1,\ldots,x_m)\mapsto (f_1,\ldots,f_{m+l})$ is defined as 
$Q_f=\OO_m/f^*\Ideal_{m+l}$, where $\OO_m$ is the ring of holomorphic function germs at $(\C^m,0)$ and $\Ideal_{m+l}$ is the maximal ideal of $\OO_{m+l}$. We will be only interested in finite germs, that is, when this algebra is finite dimensional and can be presented as $Q_f=\C[[x_1,\ldots,x_m]]/(f_1,\ldots,f_{m+l})$. It is a theorem of Mather \cite{mather4} that two germs in $\E(m,m+l)$ are $\K(m,m+l)$-equivalent if and only if their local algebras are isomorphic. Hence a (commutative, finite dimensional, local) algebra $Q$ as well as $m$ and $l$ determine a monosingularity $\eta(Q,m,l)$ (unless this set is empty).

\begin{rem}
    In practice, we can replace the vector space $\E(m,m+l)$ of germs, and the group $\K(m,m+l)$ of germs with their $N$-jets ($N \gg 0$) to obtain $\E_N(m,m+l)$ and $\K_N(m,m+l)$. In this way we obtain an algebraic group acting on a finite dimensional vector space. Our constructions and results do not depend on $N$ as long as $N$ is large enough, hence by abuse of notation we will not write the subscript~$N$.
\end{rem}

In most of our considerations the $m$-dependence will be irrelevant, hence we identify singularities $\eta(Q,m,l)$ for different $m$'s, and denote the obtained equivalence class by $\eta(Q,l)$, or simply by $\eta(Q)$ if $l$ is clear from the context. When no confusion arises, we can also simply write $Q$ for the singularity $\eta(Q)$. The notation of algebras
\[
    A_k=\C[x]/(x^{k+1}), \qquad I_{a,b}=\C[x,y]/(xy,x^a+y^b), \qquad \III_{a,b}=\C[x,y]/(x^a,xy,y^b)
\]
(for $k\geq 0$, $a\geq b\geq 2$) is standard in singularity theory.
The codimension of $\eta(Q,m,l)$ in $\E(m,m+l)$ (when the former is not empty) is independent of $m$ and is a linear function $\mu l + b$ of $l$. Here $\mu$ and $b$ are numerical invariants of the algebra $Q$; in fact $\mu=\dim_{\C}Q-1$. For example 
\[
\begin{array}{rcll}
\codim(A_k \subset \E(*,*+l)) & = & kl+k & \text{for }l\geq 0, \\
\codim(I_{a,b} \subset \E(*,*+l) ) & = & (a+b-1)l+(a+b) & \text{for }l\geq 0, \\
\codim(\III_{a,b} \subset \E(*,*+l) ) & = & (a+b-2)l+(a+b) & \text{for }l\geq 1. 
\end{array}
\]

\subsection{Multisingularities}
Consider singularities for a fix $l\geq 0$. 

\begin{df} \ 
\begin{itemize}
     \item A T-multisingularity (target-multisingularity) $\multi$ is a finite multiset of singularities. We will use intuitive notation, for example, if $\eta_1$ and $\eta_2$ denote singularities, then $(\eta_1^5,\eta_2^2)$ or simply $\eta_1^5\eta_2^2$ will denote the multiset containing these two monosingularities with multiplicities 5 and 2.
	\item An S-multisingularity (source-multisingularity) is a T-multisingularity $\multi$ with a distinguished element, i.e. a pair $(\eta_1,\multi)$, such that $\eta_1$ is a singularity, $\multi$ is a T-multisingularity and $\eta_1\in \multi$. We usually omit $\eta_1$ in notation, when it is clear from the context.
\end{itemize}
\end{df}

	The empty set is a T-multisingularity but not an S-multisingularity. A monosingularity can be regarded as both an S- and a T-multisingularity.

\begin{df}
	For a T-multisingularity $\multi=(\eta_1^{a_1},\dots,\eta_k^{a_k})$ we define a number $$|\Aut(\multi)|=a_1!\cdot a_2!\cdot\ldots\cdot a_k!,\qquad |\Aut(\varnothing)|=1.$$
	For an S-multisingularity $(\eta_1,\multi)$ we set $|\Aut(\eta_1,\multi)|=|\Aut(\multi\setminus \eta_1)|\,,$ i.e. for $\multi=(\eta_1^{a_1},\dots,\eta_k^{a_k})$ we have
		$$|\Aut(\eta_1,\multi)|=(a_1-1)!\cdot a_2!\cdot\ldots\cdot a_k!\,.$$
\end{df}

\begin{df} Let $\mzeta$ and $\multi$ be T-multisingularities.
\begin{itemize}
    \item We write $\mzeta\subset \multi$ when $\mzeta$ is a submultiset of $\multi$.
    \item We write  $\multi+\mzeta$ for the disjoint union of $\multi$ and $\mzeta$.
\end{itemize}
\end{df} 
\begin{ex}
    Let $\multi=A^2_0A_1$ and $\mzeta=A_0^2$, then $\multi+\mzeta=A^4_0A_1$.
\end{ex}
 \begin{df} [Multisingularity induced by subset]
  Let $\multi$ be a T-multisingularity. Introduce an order on it $\multi=\{\eta_1,\eta_2,\ldots,\eta_k\}$, i.e. a bijection $\multi \simeq [k]$. A subset $I\subset [k]$ induces a T-multisingularity
  $$\multi_I=\{\eta_i\}_{i\in I}\,.$$
 For an S-multisingularity we choose a bijection $\multi \simeq [k]$, such that the distinguished element corresponds to~1. Then subsets containing 1 induce S-multisingularities.   
 \end{df}
 \begin{ex}
 	Consider the S-multisingularity $\multi=(A_0,A_0^3A_1A_2)$ with an order $\multi=\{A_0,A_1,A_0,A_2,A_0\}$. We have
 	$$
    \multi_{\{2,3\}}=\multi_{\{2,5\}}=A_0A_1, \qquad 
    \multi_{\{1,2,4\}}=(A_0,A_0A_1A_2)\ (\text{as an S-multisingularity}). 
    $$
 \end{ex}

\subsection{Multisingularity loci}

The contact group contains the complex holomorphic reparametrization group of the source and the target---the so called `right-left' group. Hence, the following concepts are well defined. 

\begin{df}
Let $f:M\to N$ be a map of relative dimension $l$.
	\begin{itemize}
		\item For $x\in M$ we write $\eta_x$ for the singularity induced by the germ $f:(M,x)\to (N,f(x)).$
		\item Suppose that the preimage of $y\in N$ is finite. We write $\multi_y$ for a T-multisingularity induced by germs of $f$ on all preimages of $y$.
		\item For a T-multisingularity $\multi$ we consider a subset $\Tlocus{\multi}{f}\subseteq N$ (`T-multisingularity locus') defined by
		$$
		\Tlocus{\multi}{f}=\{y\in N|\multi_y=\multi\}\,.
		$$
		\item For an S-multisingularity $(\eta_1,\multi)$ we consider a subset $\Slocus{(\eta_1,\multi)}{f}=\Slocus{\multi}{f}\subseteq M$ (`S-multisingularity locus') defined by 
		$$
		\Slocus{(\eta_1,\multi)}{f}=\{x\in M|\eta_x=\eta_1\,, \multi_{f(x)}=\multi\}\,.
		$$
	\end{itemize}
\end{df}

\subsection{Prototypes of singularities, multisingularities}
\label{sec: proto}
Let $Q=\C[[x_1,\ldots,x_a]]/(g_1,\ldots,g_{a+l})$ be a presentation of a finite dimensional algebra with minimal number $a$ of generators, and exactly $l$ more polynomial relations than generators. We call the germ
\[
g:(\C^a,0)\to (\C^{a+l},0), \qquad 
(x_1,\ldots,x_a)\mapsto (g_1,\ldots,g_{a+l})
\]
the genotype for $Q$ and $l$. The prototype germ for $Q$ and $l$ is defined as
\[
p:(\C^a\times V, 0) \to (\C^{a+l}\times V,0), \qquad 
(x,\phi) \mapsto (g(x) + \phi(x), \phi),
\]
where 
$V$ is a linear complement of 
\[
t_g(\theta_a) + g^*(\Ideal_{a+l})\theta_g \quad \text{ in } \quad \Ideal_{a}\!\theta_g.
\]
Here we used the following standard notions of singularity theory: 
\[
\theta_a=\E(a,1) \text{ (algebra)}, \quad \Ideal_a=\text{its maximal ideal}, \quad \theta_g=\text{space of vector fields along $g$},
\]
\[
t_g:\theta_a\to \theta_g, \qquad t_g(h)=dg \circ h.
\]
\begin{ex}\label{ex:I22,l=1,prototype}
    Let $Q=\C[x,y]/(x^2,y^2)$ (isomorphic to $I_{22}$) and $l=1$. The genotype is $g:(x,y)\mapsto (x^2,y^2,0)$. For a basis of $V$ we can choose 
    \begin{multline*}
    u_1:(x,y)\mapsto (y,0,0), \quad
    u_2:(x,y)\mapsto (0,x,0), \\
    u_3:(x,y)\mapsto (0,0,x), \quad
    u_4:(x,y)\mapsto (0,0,y), \quad
    u_5:(x,y)\mapsto (0,0,xy),
    \end{multline*}
    and for the prototype we obtain $p:(\C^7,0)\to (\C^8,0)$
\begin{equation*}\label{eq:I22,l=1,prototype}
p:(x,y,u_1,u_2,u_3,u_4,u_5)\mapsto (x^2+u_1y,y^2+u_2x,u_3x+u_4y+u_5xy,u_1,u_2,u_3,u_4,u_5).
\end{equation*}
\end{ex}

Let $\multi=(\eta_1,\ldots,\eta_k)$ be a non-empty T-multisingularity. 
Let $p_i:(\C^{b_i},0)\to (\C^{b_i+l},0)$ be the corresponding prototypes for $i=1,\ldots,k$. Let $b=\sum_i {b_i} + (k-1)l$, and define the {\em prototype of the multisingularity $\multi$} to be the $k$-multigerm 
\[
p:(\C^b,0) \sqcup  (\C^b,0) \sqcup \ldots \sqcup (\C^b,0) \ \to\  (\oplus_{i=1}^k \C^{b_i+l},0)
\]
that, on the $j$'th component of the domain, is defined  by 
\[
p|_{j\text{th } \C^b}: \quad
\left( \id_{\C^{b_1+l}} \oplus \ldots \oplus \id_{\C^{b_{j-1}+l}} \right)
\oplus \  p_j \ \oplus 
\left( \id_{\C^{b_{j+1}+l}} \oplus \ldots \oplus \id_{\C^{b_k+l}} \right). 
\]
\begin{ex}
Let $l=1$. The prototype of the multisingularity $A_0^3$ is
\[
(\C^2_{x_1,y_1},0) \sqcup (\C^2_{x_2,y_2},0) \sqcup (\C^2_{x_3,y_3},0) \to (\C^3,0),
\]
given by 
\[
  (x_1,y_1)\mapsto (0,x_1,y_1), \ 
  (x_2,y_2)\mapsto (x_2,0,y_2), \ 
  (x_3,y_3)\mapsto (x_3,y_3,0).   
\]
The prototype of the multisingularity $A_1A_0^2$ is
\[ 
(\C^4,0) \sqcup (\C^4,0) \sqcup (\C^4,0) \to (\C^5,0), 
\]
given by 
\begin{multline*}
    (x_1,x_2,x_3,x_4)\mapsto (x_1^2,x_1x_2,x_2,x_3,x_4),\\
    (x_1,x_2,x_3,x_4)\mapsto (x_1,x_2,x_3,0,x_4),\ 
    (x_1,x_2,x_3,x_4)\mapsto (x_1,x_2,x_3,x_4,0),
\end{multline*}
on the three components of the domain.
\end{ex}

\begin{df}
The S-codimension and T-codimension ($\scodim(\multi)$ and $\tcodim(\multi)$) of a multisingularity $\multi$ are defined to be the dimensions of the source and target spaces of its prototype. 
\end{df}

For a monosingularity, the notion $\scodim$ coincides with its codimension in $\E(m,m+l)$. By definition, we have $\tcodim(\multi)=\scodim(\multi)+l$ for any $\multi$, and 
\[
\scodim(\eta_1,\eta_2,\ldots,\eta_r)=\sum^r \scodim(\eta_i) + (r-1)l,
\qquad
\tcodim(\eta_1,\eta_2,\ldots,\eta_r)=\sum_{i=1}^r \tcodim(\eta_i).
\]
For example, for $l=1$, we have
\begin{equation}\label{eq:codim_examples}
\scodim(A_0)=0, \quad
\scodim(A_1)=2, \quad
\scodim(I_{22})=7, \quad
\scodim(A_0^2A_1I_{22})=12.
\end{equation}

\subsection{Nice dimensions, Mather singularities} \label{sec:Mather}
\cite{mather3, mather4, TPP}
For $l\geq 1$ we define the {\em Mather bound}
    \[
    M(l)=\begin{cases} 6l+8 & \text{if } l=1,2,3, \\ 6l+7 & \text{if } l\geq 4. \end{cases}
    \]
It is a fact that there are only finitely many algebras $Q$ that are local algebras of singularities with codimension $\leq M(l)$ in $\E(m,m+l)$, for any $m$.
Moreover, for any $k\le M(l)$ the subset of $\E(m,m+l)$ corresponding to algebras with codimension at least $k+1$ is of codimension at least $k+1$.
For large $m$ the list of algebras with codimension $\leq M(l)$
is the same: we will call them the Mather algebras for $l$. For $l=1$ there are 32 Mather algebras, and for $l\geq 18$ there are 48 Mather algebras (the same 48). The complete list of algebras that are Mather algebras for some $l$ contains 55 algebras, see \cite{TPP}.

Multisingularities (either S- or T-) with S-codimension at most $\leq M(l)$ are called Mather multisingularities. All the multisingularities in \eqref{eq:codim_examples} are Mather multisingularities, as $M(1)=14$. For small $l$ the Mather T-multisingularities are listed on \cite{TPP}. For $l=1$ there are 265 Mather T-multisingularities, for $l=2$ there are 185.

\subsection{Important classes of maps} \label{sec:CC_l}

Universal counting formulas are expected to only be valid for maps that are stable under perturbations. Namely, define a map $f$ {\em stable}, if any closeby map $g$ (in the appropriate topology) $f$ and $g$ are right-left equivalent: $g= \phi \circ f\circ \psi^{-1}$ for diffeomophisms $\psi, \phi$ of the domain and co-domain. This class of maps is reasonable over the reals, but not over the complexes (think of, for example, maps from a compact manifold to $\C^n$). Experience shows that for complex maps the `target-local version' of perturbation and stability are the right substitutes:

\begin{df}[Def. 3.4 in \cite{MNB}]
 An unfolding of a map germ $g:(\C^m,S)\to (\C^{m+l},0)$ is a germ $G:(\C^{m}\times \C^d , S \times \{0\}) \to (\C^{m+l} \times \C^d,0)$ of the form $(\tilde{g}(x,u),u)$ with $\tilde{g}(x,0)=g$. The unfolding $G=g \times \id_{\C^d}$ is called trivial. If all unfoldings of a germ are equivalent (via the natural equivalence on unfoldings) to a trivial unfolding, then it is called locally stable.
\end{df}

\begin{df}[Def. 4.4 in \cite{MNB}]
A map $f:M^m \to N^{m+l}$ is locally stable if all its induced germs $f: (M,f^{-1}(y)) \to (N,y)$ are locally stable. The class of locally stable maps of relative dimension $l$ will be called $\CCv_l$. 
\end{df}

We will need some other classes of maps with more restrictions on the induced multigerms.

\begin{df} \label{def:CMa}
Let $0 \leq k \leq M(l)$.  
Consider maps $f \colon M^m \to N^{m+l}$ such that, for every $y \in N$, the germ  
$
  f \colon (M, f^{-1}(y)) \to (N, y)
$
is (right-left equivalent to) a trivial unfolding of the prototype of some multisingularity $\multi$ with $\scodim(\multi) \leq k$.  
The class of such maps will be denoted by $\CCv_l^{\Ma,k}$.  
In the case $k = M(l)$, that is, when all Mather multisingularities are allowed, we simply write
$
  \CCv_l^{M\!a} := \CCv_l^{\Ma,M(l)}.
$
\end{df}

Maps in $\CCv_l^{\Ma}$ are more manageable: for example, it is clear that for $f\in\CCv_l^{\Ma}$ we have
\[
\codim\left( \Slocus{\zeta,\multi}{f} \subset M \right) = \scodim(\multi), 
\qquad
\codim\left( \Tlocus{\multi}{f} \subset N \right) = \tcodim(\multi),
\]
independently of $\zeta\in \multi$. For example, for $l=1$, we have
\[
\codim\left( \Slocus{\zeta,A_0^2A_1I_{22}}{f} \subset M\right)=12,  
\quad
\codim\left( \Tlocus{A_0^2A_1I_{22}}{f} \subset N\right)=13, 
\]
for $\zeta$ being either $A_0$, $A_1$ or $I_{22}$.

The theory of locally stable maps outside of the nice dimension range is more subtle. However, it is a theorem of Mather that if $m\leq M(l)$ then locally stable maps are exactly those in $\CCv_l^{\Ma}$.  

\bigskip

In this paper, we aim to provide universal counting formulas for maps $f:M\to N$ that involve the push-forward map $f_*$ in cohomology. Hence, we restrict our map classes to those 
\[
\CC_l \subset \CCv_l, \qquad 
\CC_l^{\Ma,k} \subset \CCv_l^{\Ma,k}, \qquad
\CC_l^{\Ma} \subset \CCv_l^{\Ma}
\]
containing only finite morphisms. Then they are proper, hence $f_*$ is defined.

Besides finite maps between compact manifolds, these classes include, for example, prototypes of Mather singularities even {\em considered equivariantly}---that we will explain now.

Let $G$ be a complex algebraic group (typically a torus $(\C^*)^r$) with two representations $\rhoS$ and $\rhoT$ on $\C^m$ and $\C^{m+l}$. We say that the group $G$ (more precisely the triple $(G, \rhoS, \rhoT)$) is a symmetry of the germ $f:(\C^m,0)\to (\C^{m+l},0)$, if for a representative of $f$ (that we also denote by $f$) we have
\[
\rhoT(g) \circ f \circ \rhoS(g^{-1}) = f \qquad \text{ for all } g\in G.
\]
Applying the Borel construction ($X \mapsto B_GX:=EG \times_G X$ for a contractible space $EG$ with a free $G$-action) to this situation we obtain a map
\begin{equation}\label{eq:BGf}
B_Gf: B_G\C^m \to B_G\C^{m+l}.
\end{equation}
This map is a fibration over the classifying space $BG$ of $G$ with fiber $f:\C^m \to \C^{m+l}$. In each fiber the map is right-left equivalent to $f$. In practice, when we are interested in a cohomological identity in a concrete degree, we can approximate the classifying space with a finite dimensional manifold (eg. for $G=\C^*$ we can use $\CP^N$ instead of $\CP^{\infty}$ for large $N$). This way the map $B_Gf$ is a map between manifolds, and is locally just a trivial unfolding of $f$. 

If $f\in \CC_l$ and the symmetry satisfies that the equivariant Euler class $\eu(\rho^S) \in \coh^*(BG)$ is not 0, then the cohomology pushforward $(B_Gf)_*: \coh^*(B_G\C^m) = \coh^*(BG) \to \coh^*(B_G\C^{m+l})=\coh^*(BG)$ is defined as multiplication by $\eu(\rhoT)/\eu(\rhoS) \in \coh^*(BG)$, cf. Example~\ref{ex:I22_classes}.


Maps of the type \eqref{eq:BGf} have been a central tool in Thom polynomial theory since \cite{rrtp}. For brevity, we often describe a map $B_Gf$ by saying that we consider $f$ “equivariantly” or “in equivariant cohomology.”

\section{The Thom principle}
\label{sec:ThomPrinciple}

Let $\cl(\Sigma \subset M)\in \coh^\bullet(M)$ be a cohomological invariant (a ``characteristic class'') associated to (not necessarily closed) subvarieties $\Sigma$ of smooth manifolds $M$. Suppose $\cl$ satisfies a suitable consistency condition with respect to pullbacks. The simplest example is {\em ``fundamental class of the closure''}, $\cl(\Sigma\subset M)=[\overline{\Sigma}]$. 
The following intuitive statement we call the Thom principle for $\cl$: 

\medskip

\noindent{\bf Thom Principle for Multisingularities.}  
\textit{For any S-multisingularity (respectively, T-multisingularity) $\multi$, there exists a universal expression $P$---a multivariable polynomial or formal power series---such that, for every `suitably nice map' $f: M \to N$ between manifolds, the class $\cl(\Slocus{\multi}{f} \subset M)$ (respectively, $\cl(\Tlocus{\multi}{f} \subset N)$ is obtained by evaluating $P$ on the characteristic classes associated with $f$.}

\medskip

For detailed discussions of various flavors of the Thom principle and its history, we refer the readers to the recent survey papers \cite{OhmotoSurvey, primer} and references therein. Here we just make a few remarks:
\begin{itemize}
    \item The Thom principle for monosingularities, and $cl$ being the fundamental cohomology class of the closure of $\Sigma$, holds. The polynomial $P$ in this case can be interpreted as the $\K(m,n)$-equivariant fundamental class of $\overline{\eta}\subset \E(m,n)$---after careful treatment of the infinite dimensional and non-compact groups and spaces. This is the field of the classical Thom polynomials.
    \item The nature of the subjects is calculational. Often even the vague Thom principle is sufficient to calculate the universal polynomial $P$ (assuming its existence), and then these universal formulas can be used in enumerative geometry---with ad hoc arguments about their validity.
    \item In all instances of the Thom principle, the concept of ``suitably nice map'' is challenging. 
    \item There is an alternative approach to the Thom principle, that avoids (or relocates) the discussion on ``suitably nice maps''. Namely, redefining the meaning of the class ``$\cl(\Tlocus{\multi}{f})$'' in such a way that the Thom principle holds {\em for all maps}, and then discussing for which maps does the new definition agree with the geometric definition. This program is carried out for the fundamental classes of T-multisingularities in \cite{TOmult}. 
\end{itemize}

In this paper we study the Thom principle for multisingularities, for the Segre-Schwartz-MacPherson characteristic class $\cl(\Sigma\subset M)=\ssm(\Sigma\subset M)$.

\section{Characteristic classes}
\subsection{The \texorpdfstring{$\csm$}{csm} and \texorpdfstring{$\ssm$}{ssm} classes}\label{s:ssm}
We present a brief introduction to the Chern–Schwartz–\-Mac\-Pherson ($\csm$) class and the Segre–Schwartz–MacPherson ($\ssm$) class. For a more extensive treatment as well as discussion of the broader topic of characteristic classes of singular varieties, see the survey \cite{SY} and monographs \cite{SchLectures}, \cite[Ch.~5-7]{BrasseletBook}.

For a variety $X$, let $F(X)$ denote the $\QQ$-vector space of constructible functions. We may treat $F$ as a functor from the category of varieties (with morphism being the proper maps) to the category of $\QQ$-vector spaces. 
In \cite{macpherson} MacPherson constructed a unique natural transformation from $F$ to the Borel-Moore homology $c_*$,
such that for a compact smooth variety $M$ we have
$$c_*(\id_M)=PD(c_\bullet(TM)) \in \coh^{BM}_\bullet(M)\,.$$
Here $c_\bullet(TM)$ denotes the total Chern class of the tangent bundle and $PD$ is Poincair\'e duality. 
The $\csm$ and $\ssm$ classes of a constructible subset $X$ of a smooth variety $M$ are defined as
$$
\csm(X\subset M):=PD(c_*(1_X))\in \coh^\bullet(M)\,,\qquad
\ssm(X\subset M):=\frac{\ssm(X\subset M)}{c_\bullet(TM)}\in \coh^\bullet(M)\,.\qquad
$$
We omit $M$ in the notation when it is obvious from the context. Here are some basic properties.

\begin{pro} \label{pro:csm}
    Let $M$ and $N$ be smooth varieties.
    \begin{enumerate}
        \item Let $X$ and $Y$ be disjoint constructible subsets of $M$. Then
        $$\ssm(X\cup Y)=\ssm(X)+\ssm(Y)\,.$$
        \item Let $X\subset M$ and $Y\subset N$ be constructible subsets. Then
        $$\ssm (X\times Y\subset M\times N)=\ssm (X\subset M)\boxtimes\ssm (Y\subset N)\,.$$
        \item Let $X\subset M$ be a constructible subset and $U\subset M$ an open subset. Then
        $$\ssm (X\subset M)_{|U}=\ssm(X\cap U\subset U)\,. $$
        \item Let $X \subset M$ be an irreducible subvariety of codimension $a$. Then 
        $$\ssm(X)_{|r}=
        \begin{cases}
            0 &\text{ for } r<a\,, \\
            [X] &\text{ for } r=a\,.
        \end{cases}$$
        Here $|_r$ denotes the restriction to $\coh^r(M)$ and $[X]$ denotes the fundamental class of $X$.
        \item  Let $X\subset M$ be a constructible subset. Suppose that the variety $M$ is projective. Then the $\csm$ class of $X$ determines the Euler characteristic of $X$ 
        $$\pi_*(\ssm(X)\cdot c_\bullet(TM))=\pi_*(\csm(X))=\chi(X)\,,$$
        where $\pi$ denotes the unique map from $M$ to a point. 
    \end{enumerate}   
\end{pro}
 For a variety equipped with an action of an algebraic group $G$, the $G$-equivariant version of the MacPherson transformation and the $\csm$ class were introduced by Ohmoto in \cite{OhmotoCamb}. The properties listed in Proposition \ref{pro:csm} generalize to this equivariant setting. 
 In Section \ref{s:interpolation} we will need the following result.
 \begin{pro}[{Special case of \cite[Thm. 20]{WeberCSM}}] \label{thm:weber}
     Let $V$ be a finite dimensional vector space equipped with a linear action of a torus $\TT$. Suppose that $0$ is an isolated fixed point. Let $X\subset V$ be a constructible subset such that $0\notin X$. Then the top degree part of the class $\csm(X)$ vanishes, i.e.
     $$ \csm(X)_{|\dim V}=0\in \coh^{*}_\TT(V)  \,, $$
     where $|_{\dim V}$ denotes the component in $\coh^{\dim V}_\TT(V)$.
 \end{pro}

In certain situations, interpolation properties of this kind --- together with normalization and support axioms --- uniquely determine the $\csm$ class. This viewpoint originates in \cite{MO} with the introduction of characteristic classes called {\em stable envelopes} and their relation to $\csm$ classes established in \cite{RV, FRcsm, AMSS_csm}. In \cite{RR_ssmTp}, this stable–envelope–inspired framework was used to compute SSM–Thom polynomials of monosingularities. In Section~\ref{s:interpolation}, we extend this approach to multisingularities.

\begin{rem}
Proposition~\ref{thm:weber} demonstrates how the hierarchy of strata, ordered by closure containment, can interact with characteristic classes. In the forthcoming work~\cite{JKRRhierarchy}, we will build on a similar such result from~\cite{FeherPatakfalvi}, together with Thom polynomials of multisingularities, to establish a new and previously unknown hierarchy.
\end{rem}
 
\subsection{Chern and Landweber-Novikov classes associated with maps}
For partitions $\lambda$ we use both the traditional notation 
$
\lambda=(\lambda_1\geq \lambda_2 \geq \ldots \geq \lambda_r),
$
and the `multiplicity' notation  
$
\lambda=(1^{a_1},2^{a_2},\dots)
$
meaning that $a_i$ of the parts of $\lambda$ are $i$. E.g. $\lambda=(4,3,3,2,1,1,1)=(1^32^13^24^1)$.
\begin{df}
    For a smooth map $f:M^m\to N^{m+l}$ let $T_f=f^*TN-TM$ be its relative tangent bundle.
    \begin{enumerate}
        \item  Chern classes $c_\bullet(f):=c_\bullet(T_f)$ are defined by the formula 
        $$1+c_1(f)+\dots=\frac{1+f^*c_1(TN)+\dots+f^*c_n(TN)}{1+c_1(TM)+\dots c_m(TM)} \in   \coh^\bullet(M)\,.$$
        For a partition $\lambda$ we consider the class $c_\lambda(f)=\prod_{i=1}^kc_{\lambda_i}(T_f)$.
        \item If $f_*$ is defined in cohomology, the Landweber-Novikov classes $s_\lambda(f):=s_\lambda(T_f)$ are 
        $$s_{\lambda}(f):=f_*c_\lambda(f) \in \coh^\bullet(N).$$ 
    \end{enumerate}
\end{df}

Let $\cvar=(c_1,c_2,\dots)$ denote a set of variables indexed by natural numbers and $\svar=(s_\lambda)$ a set of variables indexed by partitions.  We consider the rings of power series in these variables: $\QQ[[\cvar,\svar]]$ and $\QQ[[\svar]]$. There are two natural gradings on these rings.  The cohomological grading is given by 
$$\deg(c_t)=t\,, \qquad \deg(s_\lambda)=l+|\lambda|=l+a_1+2a_2+\dots\,,$$
where $\lambda=(1^{a_1},2^{a_2},\ldots)$.
When we say that an element is homogeneous we refer to this grading. The second grading is the $\svar$-degree $\deg_s$---it is the degree in $\svar$-variables, i.e. 
$\deg_s(c_i)=0$, $\deg_s(s_\lambda)=1$.
\begin{df} \label{df:substitution}
	Let $f:M\to N$ be a map in $\CC_l$.
	\begin{enumerate}
		\item For a power series $P\in \QQ[[\svar]]$ we define the class
		$P(f) \in \coh^\bullet(N)$
		by the substitution
		$$s_{\lambda} \to s_{\lambda}(f)\,.$$
        \item For a power series $P\in \QQ[[\cvar,\svar]]$ we define the class
		$P(f) \in \coh^\bullet(M)$
		by the substitution
		$$c_t \to c_t(f)\,, \qquad s_{\lambda} \to f^*s_{\lambda}(f)\,. $$
	\end{enumerate}
	Both substitutions are homogeneous with respect to the cohomological grading. 
\end{df}

\begin{pro} \label{pro:unique}
	For every polynomial $A\in \QQ[[\cvar]]$ there exists a map $f\in \CC_l$ such that $A(f) \neq 0$. The same holds for a polynomial $B\in \QQ[[\svar]]$, or $C\in \QQ[[\cvar,\svar]]$.
\end{pro}
\begin{proof}
    See Propositions \ref{pro:uniqueC1}, \ref{pro:uniqueS1'} and \ref{pro:uniqueSC} in Appendix \ref{ap:unique}.
\end{proof}
\begin{rem}
    It is a well known fact that there is no algebraic relation between the Chern and Ladweber-Novikov classes that holds for all maps between smooth varieties. The above proposition states that the same is true if we restrict ourselves to the smaller class of maps $\CC_l$. This is a folklore result and we were unable to find a reference. For the sake of completeness we prove it in Appendix \ref{ap:unique}.
\end{rem}

\subsection{Characteristic classes of disjoint union} \label{s:union}
\begin{df}
	Let $f:M_1\to N_1$ and $g:M_2\to N_2$ be maps between varieties. We consider the disjoint union maps
	\begin{align*}
		f\sq g&= (f\times \id_{N_2}) \sqcup (\id_{N_1} \times g): (M_1\times N_2)\sqcup (N_1\times M_2) \to N_1\times N_2\,, \\
		\qquad f^{(k)}&=
		f\sq \cdots \sq f \,.
	\end{align*}
\end{df}
\begin{rem}
    The usage of maps $f^{(k)}$ to study Thom polynomials of multisingularities was proposed in \cite[Rem. 3.5]{KazaMulti}.
\end{rem}
\begin{rem}
    The construction of the prototype of a multisingularity from the prototypes of monosingularities (see Section \ref{sec: proto}) is based on the $\sq$ operation.
\end{rem}
\begin{rem}
    Let $G_1$ and $G_2$ be algebraic groups.
    Suppose that the map $f$ is $G_1$-equivariant and $g$ is $G_2$-equivariant. Then the map $f\sq g$ is $G_1\times G_2$-equivariant. When the groups are the same $G:=G_1=G_2$ then we can consider the diagonal action $G$-action.
\end{rem}
\begin{pro}
	Let $f_1$ and $f_2$ be maps in $\CC_l$. Then the map $f_1\sq f_2$ is also in $\CC_l$.
\qed
\end{pro}
\begin{proof}
    This follows from \cite[Thm. 3.3]{MNB}.
\end{proof}

\begin{pro}\label{pro:sum1}
	Let $f:M_1\to N_1$ and $g:M_2\to N_2$ be maps in $\CC_l$. Let $\pi_{1},\pi_2$ denote the projections from $N_1\times N_2$ onto the factors and $\tilde{\pi}_{1}$, $\tilde{\pi}_2$ denote the projections from $M_1\times N_2$ onto the factors. For any T-multisingularity $\multi$, or S-singularity $(\eta, \multi)$ we have 
	\begin{align*}
	\ssm (\Tlocus{\multi}{f\sq g})&=\sum_{\multi_1+\multi_2=\multi} \pi_1^*\ssm (\Tlocus{\multi_1}{f})\cdot \pi_2^*\ssm (\Tlocus{\multi_2}{g})\,, \\
	\ssm (\Slocus{\eta, \multi}{f\sq g})|_{M_1\times N_2}&=
	\sum_{\multi_1+\multi_2=\multi\,, \eta\in\multi_1} \tilde{\pi}_1^*\ssm (\Slocus{\eta,\multi_1}{f})\cdot \tilde{\pi}_2^*\ssm (\Tlocus{\multi_2}{g})\,.
	\end{align*}
\end{pro}
\begin{proof}
	Lemma follows from the decompositions
	$$\Tlocus{\multi}{f\sq g}=\bigsqcup_{\multi_1+\multi_2=\multi} \Tlocus{\multi_1}{f}\times \Tlocus{\multi_2}{g}\,, \qquad \Slocus{\multi}{f\sq g}=\bigsqcup_{\multi_1+\multi_2=\multi\,, \eta\in\multi_1} \Slocus{\multi_1}{f}\times \Tlocus{\multi_2}{g}$$
	and properties of the $\ssm$ class, see Section \ref{s:ssm}. 
\end{proof}
\begin{cor} \label{cor:sum}
	Consider the situation form Proposition \ref{pro:sum1}. Suppose that the maps $f$ and $g$ have the same target $N$, i.e. $N:=N_1=N_2$. Let $\Delta:N\to N\times N$ be the diagonal map and $\Gamma:M_1\to M_1\times N$ the graph of $f$. We have
	\begin{align*}
		\Delta^*\ssm (\Tlocus{\multi}{f\sq g})&=\sum_{\multi_1+\multi_2=\multi} \ssm (\Tlocus{\multi_1}{f})\cdot\ssm (\Tlocus{\multi_2}{g})\,, \\
		\Gamma^*\ssm (\Slocus{\multi}{f\sq g})&=\sum_{\multi_1+\multi_2=\multi\,, \eta\in\multi_1} \ssm (\Slocus{\multi_1}{f})\cdot f^*\ssm (\Tlocus{\multi_2}{g})\,.
	\end{align*}
\end{cor}
\begin{pro}\label{pro:sum2}
	Let $f:M_1\to N_1$ and $g:M_2\to N_2$ be maps in $\CC_l$. We have:
	\begin{align*}
		&s_\lambda(f\sq g)=\pi_1^*s_\lambda(f)+ \pi_2^*s_\lambda(g) \\
		&c_k(f\sq g)_{|M_1\times N_2}=\tilde{\pi}_1^*c_k(f)\,,
	\end{align*}	
	where $\tilde{\pi}_1:M_1\times N_2\to M_1$ is the standard projection.
\end{pro}
\begin{proof}
    The statement follows from isomorphism of relative tangent bundles $T_{f\sq g}|_{M_1\times N_2}=\tilde{\pi}^*T_f$.
\end{proof}

\section{SSM-Thom polynomials for multisingularities} \label{s:conjecture}
In the rest of the paper we assume that $l\ge 1$.
\begin{con} \label{con:SSM}
	The Thom principle holds for $\ssm$ classes of multisingularities.
	\begin{description}
		\item[T] For every T-multisingularity $\multi$ there exists a power series $\ThT{\multi} \in \QQ[[\svar]]$, called target SSM-Thom polynomial of $\multi$, such that for every map $f:M\to N$ in $\CC_l$ we have
		$$\ThT{\multi}(f)=\ssm(\Tlocus{\multi}{f}) \cdot |\Aut(\multi)| \in \coh^\bullet(N)\,. $$
		\item[S] For every S-multisingularity $\multi$ there exists a power series $\ThS{\multi} \in \QQ[[\cvar,\svar]]$, called source SSM-Thom polynomial of $\multi$, such that for every map $f:M\to N$ in $\CC_l$ we have
		$$\ThS{\multi}(f)=\ssm(\Slocus{\multi}{f}) \cdot |\Aut(\multi)| \in \coh^\bullet(M)\,. $$
	\end{description}
\end{con}
 In Section \ref{sec:induction} we prove structure theorems for target and source SSM-Thom polynomials, assuming the conjecture above. If the Thom polynomials $\ThT{\multi}$ and $\ThS{\multi}$ exist, then they are uniquely determined by the conjectured property, due to Proposition \ref{pro:unique}.
 \begin{rem}
 Conjecture \ref{con:SSM} has a fundamental class version (the lowest degree part of the $\ssm$ class). It was stated in \cite[Thm. 2.2]{KazaMulti}. Its `target version' was proved in \cite{TOmult}.
 \end{rem}
 \begin{rem}
     Conjecture \ref{con:SSM}[S] implies Conjecture \ref{con:SSM}[T]. The polynomial $\ThS{\multi}$ determines $\ThT{\multi}$. We study the relation between them in Section \ref{s:F}.
 \end{rem}
 In Mather's nice dimension range we consider a degree-cut version of Conjecture \ref{con:SSM}. In many cases we can verify this version of the conjecture.
 \begin{df}
     For a natural number $k\in \NN$ we use the notation $|_k:\coh^\bullet(-)\to \coh^k(-)$ and $|_{\le k}:\coh^\bullet(-)\to \coh^{\le k}(-)$ for the standard projection maps.
 \end{df}
\begin{con} \label{con:SSM2}
        The Thom principle holds for $\ssm$ classes of Mather multisingularities up to the the degree given by the Mather bound. Let $\multi$ be a Mather T-, or S-multisingularity.
        \begin{description}
		\item[T] Let $k\le M(l)+l$ be a degree bound.  There exists a polynomial $\ThTMa{\multi}{k} \in \QQ[\svar]$, such that for every map $f:M\to N$ in $\CC_l$ we have
		$$\ThTMa{\multi}{k}(f)=\ssm(\Tlocus{\multi}{f})_{|\le k} \cdot |\Aut(\multi)| \in \coh^{\le k}(N)\,, $$
		\item[S] Let $k\le M(l)$ be a degree bound. There exists a polynomial $\ThSMa{\multi}{k} \in \QQ[[\cvar,\svar]]$, such that for every map $f:M\to N$ in $\CC_l$ we have
		$$\ThSMa{\multi}{k}(f)=\ssm(\Slocus{\multi}{f})_{|\le k} \cdot |\Aut(\multi)| \in \coh^{\le k}(M)\,. $$
	\end{description}
\end{con}
  In Section \ref{s:interpolation}, we will specify an equivalent version of the above conjecture that can be verified using computer computations. Conjecture \ref{con:SSM2} holds in all the computed examples. We present some of the computed Thom polynomials in Appendix \ref{ap:examples}. We uploaded many more examples to the Thom polynomial portal \cite{TPP}.

\section{Structure theorems for SSM-Thom polynomials} \label{sec:induction}
To keep track of the various notions in the next sections the reader is advised to consult Figure~\ref{fig:latest}.
The proofs in this section follow the ideas of Kazarian \cite{KazaSapporo,KazaMulti}. In Sections \ref{sec:induction1} and \ref{sec:source} we present only the results, their proofs are in Sections \ref{sec:induction3} and \ref{sec:induction4}.
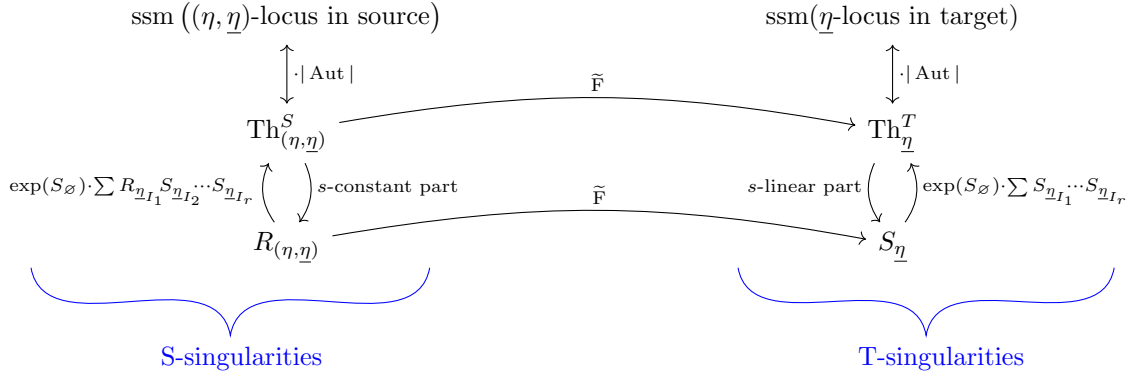
\begin{figure}
\begin{tikzcd}[remember picture,column sep=2cm,row sep=.8cm]
  \ssm\left((\eta,\multi)\text{-locus in source}\right)
 \ar[<->,d,"\cdot|\Aut|"]
 & & 
  \ssm(\multi\text{-locus in target})
 \ar[<->,d,"\cdot|\Aut|"]
 \\
 \ThS{(\eta,\multi)} \ar[->,d,bend left=30,"s\text{-constant part}" pos=.4] 
 \ar[d,<-,bend right=30,swap,"\MS\cdot \sum \KazS{\multi_{I_1}}\!\!\KazT{\multi_{I_2}}\!\!\!\cdots \KazT{\multi_{I_r}}" pos=.461] 
 \ar[->,rr,"\FF",bend left=10]
 & & 
 \ThT{\multi} \ar[->,d,bend right=30,swap,"s\text{-linear part}" pos=.4] 
 \ar[d,<-,bend left=30,"\MS\cdot \sum \KazT{\multi_{I_1}}\!\!\!\cdots \KazT{\multi_{I_r}}" pos=.4617]\\
 \KazS{(\eta,\multi)} \ar[->,rr,"\FF",bend left=10]
 & & \KazT{\multi} 
 \\ 
 \\
  \end{tikzcd}
\begin{tikzpicture}[overlay,remember picture]
\draw[blue] (-14.8,-1) to[out=-80,in=100] (-12.15,-1.9) to[out=80,in=-100] (-9.5,-1) ;
\node[blue] at (-12,-2.2) {S-singularities};
\draw[blue] (-5.4,-1) to[out=-80,in=100] (-2.9,-1.9) to[out=80,in=-100] (-0.4,-1) ;
\node[blue] at (-2.7,-2.2) {T-singularities};
\end{tikzpicture}
\caption{The interrelations among the various concepts of Section \ref{sec:induction}.}\label{fig:latest}
\end{figure}

\subsection{Target induction} \label{sec:induction1}
In this section we assume Conjecture \ref{con:SSM}[T].
\begin{thm} \label{thm:MS}
	Thom polynomial $\ThT{\varnothing}$ is an invertible element of $\QQ[[\svar]]$. There exists a linear power series $\MSlog \in \QQ[[\svar]]$ such that
	$$\exp(\MSlog)=\ThT{\varnothing}\,.$$
\end{thm}
We call the series $\MSlog \in \QQ[[\svar]]$ the Master Series.
\begin{df}
	 For a nonempty T-multi\-sin\-gu\-la\-ri\-ty $\multi$ we define a power series $\KazT{\multi} \in \QQ[[\svar]]$ as the $\svar$-linear part of the quotient $\ThT{\multi} \cdot \MS^{-1}$.
\end{df} 
\begin{rem}
    For a nonempty multisingularity $\multi$ the degree zero part of the Thom polynomial $\ThT{\multi}$ is equal to zero. This has easy direct proof, or may be deduced from the Theorem \ref{thm:indT} below. It follows that $\KazT{\multi}$ is also the $\svar$-linear part of $\ThT{\multi}$.
\end{rem}
\begin{thm} \label{thm:indT}
	Let $\multi$ be a nonempty T-multisingularity. Introduce an order on it, i.e. $\multi=\{\eta_1,\eta_2,\ldots,\eta_k\}$. We have
	$$
	\ThT{\multi}=
	\sum_{1\in I\subset [k]}
	\KazT{\multi_I}\cdot\ThT{\multi_{I'}} \in \QQ[[\svar]]\,,
	$$
	where $I'$ is the complement of $I$ in $[k]$.
\end{thm}
\begin{ex}
    For a monosingularity $\multi=\{\eta\}$ we have
	$$
	\ThT{\eta}=\exp(\MSlog)\cdot
	S_\eta\,.
	$$
    For $\multi=A^2_0$:
	$$\ThT{A^2_0}=\exp(\MSlog)\cdot\left(\KazT{A^2_0}+ \KazT{A_0}^2 \right)\,.$$
	For $\multi=A^3_0$:
	$$\ThT{A^3_0}=\exp(\MSlog)\cdot\left(\KazT{A^3_0}+ 3\KazT{A^2_0}\KazT{A_0}+\KazT{A_0}^3 \right)\,.$$
\end{ex} 
\begin{cor} \label{cor:indT}
	Theorem \ref{thm:indT} may be written in the  equivalent forms:
    \begin{description}
        \item[Closed formula] Fix a bijection $\multi\simeq [k]$. Consider  all decompositions of the set $[k]$ into nonempty subsets, do not distinguish between decompositions which vary by a permutation of subsets. Then
    $$
	\ThT{\multi}=\exp(\MSlog)\cdot
	\sum_{[k]=I_1\sqcup \dots \sqcup I_r}
	\KazT{\multi_{I_1}}\cdot\KazT{\multi_{I_2}}\cdot\ldots\cdot \KazT{\multi_{I_r}}\,.
	$$
    \item[Generating function] Let $X$ be a finite set of singularities. Denote by $X_T$ the set of T-multi\-sin\-gu\-la\-ri\-ties, consisting of singularities from $X$. To every $\eta\in X$ we associate a formal variable $t_\eta$ and to every $\multi\in X_T$ a monomial $t^\multi$, with $t^\varnothing=1$. Then
    \begin{align*}
	N_X:=\sum_{\multi\in X_T}
	\frac{\ThT{\multi}}{|\Aut(\multi)|}\cdot t^\multi=
	\exp \bigg(
	\sum_{\multi\in X_T}
	\frac{\KazT{\multi}}{|\Aut(\multi)|}\cdot t^\multi
	\bigg) \in \QQ[[\svar,\tvar]]
	\,.
    \end{align*}
    \end{description}
\end{cor}
All generating series from the above corollary can be combined into a single one.
\begin{cor}\label{cor:IndTver}
    To every singularity $\eta$ associate a formal variable $t_\eta$ and to every multisingularity $\multi$ a monomial $t^\multi$. We have
    \begin{align*}
	\sum_{\multi}
	\frac{\ThT{\multi}}{|\Aut(\multi)|}\cdot t^\multi=
	\exp \bigg(
	\sum_{\multi}
	\frac{\KazT{\multi}}{|\Aut(\multi)|}\cdot t^\multi
	\bigg) \in \QQ[[\svar,\tvar]]
	\,.
    \end{align*}
\end{cor}
\begin{rem}
    A similar power series for the classical Thom polynomials is present in the works of Kazarian \cite[Formula 9]{KazaSapporo}.
    The existence of the SSM-version was conjectured by Ohmoto \cite[Sect. 6.4]{ohmotoSMTP}.
    The new phenomenon here is the nonvanishing of the constant term $\MSlog$. The resulting Thom polynomials depend polynomially on other $\KazT{\multi}$ series, but exponentially on $\MSlog$.
\end{rem}

\subsection{Source induction} \label{sec:source}
In this section we assume Conjectures \ref{con:SSM}[S] and \ref{con:SSM}[T].
\begin{df}
	For an S-multisingularity $\multi$ we define a power series $\KazS{\multi} \in \QQ[[\cvar]]$
	as a $\svar$-degree zero part of the Thom polynomial $\ThS{\multi}$, i.e.
	$$\KazS{\multi}=\ThS{\multi}(\cvar,0)\,.$$
\end{df} 
\begin{thm} \label{thm:indS}
	Let $(\eta_1,\multi)$ be an S-multisingularity. Introduce an order on it such that the distinguished element corresponds to one, i.e. $\multi=\{\eta_1,\eta_2,\ldots,\eta_k\}$. We have
	$$
	\ThS{\multi}=
	\sum_{1\in I\subset [k]}
	\KazS{\multi_I}\cdot\ThT{\multi_{I'}}
	\in \QQ[[\cvar,\svar]]\,,
	$$
	where $I'$ is the complement of $I$ in $[k]$.
\end{thm}	
\begin{cor} \label{cor:IndS}
	Theorem \ref{thm:indS} may be rewritten in equivalent forms
    \begin{description}
        \item[Closed formula] Fix a bijection $\multi\simeq [k]$, such that the chosen singularity $\eta$ corresponds to $1$. Consider  all decompositions of the set $[k]$ into nonempty subsets, do not distinguish between decompositions which vary by a permutation of subsets. Then 
        $$
	    \ThS{\multi}=\exp(\MSlog)\cdot
	    \sum_{[k]=I_1\sqcup \dots \sqcup I_r\,, 1\in I_1}
	    \Big(\KazS{\multi_{I_1}}\cdot \prod_{j=2}^{r}\KazT{\multi_{I_j}}\Big)\,,
	    $$
        \item[Generating function] Let $X$ be a finite set of singularities with a chosen element $\eta_1$. Denote by $X_S$ the set of S-multisingularities, consisting of singularities from $X$, with chosen singularity $\eta_1$. To every $\eta\in X$ we associate a formal variable $t_\eta$ and to every $\multi\in X_S$ a monomial $t^\multi$. Then
    \begin{align*}
	M_X:=\sum_{\multi\in X_S}
	\frac{\ThS{\multi}}{|\Aut(\multi)|}\cdot t^\multi=
    \bigg(\sum_{\multi\in X_S}
	\frac{\KazS{\multi}}{|\Aut(\multi)|}\cdot t^\multi\bigg)\cdot N_X
	\,,
    \end{align*}
    where $N_X$ is the power series from Corollary \ref{cor:indT}.
    \end{description}
\end{cor}
\begin{ex}
    For a monosingularity $(\eta,\{\eta\})$ we have
	$$\ThS{\eta}(\cvar,\svar)=\exp(\MSlog)(\svar)\cdot \KazS{\eta}(\cvar).$$
	For the singularity $A_0$ the series $\KazS{A_0}(\cvar)$ starts with one, therefore we have
	$$\ThS{A_0}(0,\svar)=\exp(\MSlog)=\ThT{\varnothing}\,.$$
\end{ex}

\subsection{Proof of the target induction} \label{sec:induction3}
\begin{lemma} \label{lem:linearity}
	Let $\underline{x}$ be a set of variables and $A\in \QQ[[\underline{x},\svar]]$ a power series. Suppose that
	\begin{align} \label{w:1}
		A(\underline{x},2\cdot\svar)=2\cdot A(\underline{x},\svar) \in \QQ[[\underline{x},\svar]]\,,
	\end{align}
	where $2\cdot \svar$ denotes rescaling of all $\svar$-variables by 2. Then $A$ is $\svar$-linear.
\end{lemma}
\begin{proof}
	Let $m$ be a monomial in variables $\underline{x}$ and $\svar$ and $a_m$ the corresponding coefficient of $A$. Equation \eqref{w:1} implies that
	$(2^{\deg_s m}-2)a_m=0$. Therefore, the coefficient $a_m$ may be nonzero only when $\deg_s m=1$.
\end{proof}
\begin{lemma} \label{lem:master1}
	The constant term of the series $\ThT{\varnothing}$ is equal to 1.
\end{lemma}
\begin{proof}
	Let $f:M\to N$ be any map in $\CC_l$.  We have
	$$\ThT{\varnothing}(f)=\ssm(\Tlocus{\varnothing}{f})=\ssm(N\setminus f(M))=\ssm(N)-\ssm(f(M))=1-\ssm(f(M)) \in \coh^\bullet(N)\,.$$
	We have $l\ge 1$, so the image $f(M)$ is of codimension at least one.
	It follows that the degree zero part of $\ThT{\varnothing}(f)$ is 1.
	Substitution $s_\lambda \to s_\lambda(f)$ is an identity on the degree zero part of $\QQ[[\svar]]$.
\end{proof}
The above lemma implies that
$ \MSlog:= \log(\ThT{\varnothing})$
is a well defined power series. To prove Theorem \ref{thm:MS} we need to check that it is linear.
\begin{lemma} \label{lem:MSmult} 
	Let $f:M_1\to N$ and $g:M_2\to N$ be maps in $\CC_l$ and $\Delta:N\to N\times N$ the diagonal map.
	Then
	$$\Delta^*\ThT{\varnothing}(f\sq g)=\ThT{\varnothing}(f)\cdot \ThT{\varnothing}(g) \in \coh^\bullet(N)\,.$$
\end{lemma}
\begin{proof}
	This follows directly from Corollary \ref{cor:sum} for $\multi=\varnothing$.
\end{proof}
\begin{lemma} \label{lem:MSlin}
	The power series $\MSlog$ is linear.
\end{lemma}
\begin{proof}
	Due to lemma \ref{lem:linearity} it is enough to check that $\MSlog$ commutes with multiplication by 2, i.e.
	\begin{align} \label{w:2}
		\MSlog(2\cdot\svar)=2\cdot\MSlog(\svar) \in \QQ[[\svar]]\,.
	\end{align}
	Let $f:M\to N$ be a map in $\CC_l$ and $\Delta:N\to N\times N$ the diagonal embedding. We have
	\begin{align*}
		\ThT{\varnothing}(2\cdot\svar)(f)=\Delta^*\ThT{\varnothing}(f^{(2)})=\ThT{\varnothing}(f)^2\in \coh^\bullet(N)\,.
	\end{align*}
	The first equation follows from Proposition \ref{pro:sum2} and the second from Lemma \ref{lem:MSmult}. Applying logarithm to the above formula we obtain that equation \eqref{w:2} holds after substitution to any $f\in \CC_l$. Proposition \ref{pro:unique} implies that it holds also in $\QQ[[\svar]]$. 
\end{proof}
The rest of this section is devoted to the proof of Theorem \ref{thm:indT}. We will prove its generating function version from Corollary \ref{cor:indT}. We use notation from there. \\
Consider the power series $\N=N_X\cdot \exp(\MSlog)^{-1}$. For a map $f\in \CC_l$ we have
\begin{align}\label{w:N2}
	\N(f)=\sum_{\multi\in X_T}
	\frac{\ssm(\Tlocus{\multi}{f})}{\ssm(\Tlocus{\varnothing}{f})}\cdot t^\multi
	\in \coh^\bullet(N)[[\tvar]]
	\,,
\end{align}
The above series starts with 1, therefore it has well-defined logarithm
$$ \log(\N)=\sum_{\multi\in X_T\,, \multi\neq \varnothing}
\frac{A_\multi}{|\Aut(\multi)|} \cdot t^\multi \in \QQ[[\tvar,\svar]]\,.
$$
To prove Theorem \ref{thm:indS} we need to show that $A_\multi=\KazT{\multi}$ for any $\multi\in X_T$. It is enough to prove that the series $A_\multi$ are $\svar$-linear.
\begin{lemma} \label{lem:Nmult}
	Let $f:M_1\to N$ and $g:M_2\to N$ be maps in $\CC_l$ and $\Delta:N\to N\times N$ the diagonal map. Then
	$$\Delta^*\N(f\sq g)=\N(f)\cdot \N(g)
	\in \coh^\bullet(N)[[\tvar]]\,.$$
\end{lemma}
\begin{proof}
	  The result follows from Formula \eqref{w:N2} and Corollary \ref{cor:sum}.
\end{proof}
\begin{lemma} \label{lem:Nlin}
	The power series $\log(\N)$ is $\svar$-linear.
\end{lemma}
\begin{proof}
	We proceed analogously as in the proof of Lemma \ref{lem:MSlin}. By Lemma \ref{lem:linearity} we need to show that
	\begin{align} \label{w:4}
		\log(\N)(\tvar,2\cdot \svar)=2\cdot \log(\N)(\tvar,\svar) \in \QQ[[\tvar,\svar]] \,.
	\end{align}
	Let $f:M\to N$ be a map in $\CC_l$ and $\Delta:N\to N\times N$ the diagonal embedding. We have 
	\begin{align*} 
		\N(\tvar,2\cdot\svar)(f)=
		\Delta^*\N(f^{(2)})=\N(f)^2
		\in \coh^\bullet(N)[[\tvar]]\,.
	\end{align*}
	The first equality follows from Proposition \ref{pro:sum2} and the second from Lemma \ref{lem:Nmult}.
	Applying logarithm we obtain that equation \eqref{w:4}
	holds after substitution to any $f\in \CC_l$. Proposition \ref{pro:unique} implies that it holds also in $\QQ[[\tvar,\svar]]$. 
\end{proof}
This finishes the proof of Theorem \ref{thm:indS}. Let us note an immediate consequence of Lemma \ref{lem:Nlin}.
\begin{cor}\label{cor:Nlin}
	For an arbitrary integer $k\in \ZZ$ we have
	$$\N(\tvar,k\cdot\svar)=\N(\tvar,\svar)^k\,.$$
\end{cor}

\subsection{Proof of the source induction} \label{sec:induction4}

We prove the generating function version of Theorem \ref{thm:indS} from Corollary \ref{cor:IndS}. We use notation from there. \\
Consider the series $\M=M_X\cdot \exp(\MSlog)^{-1}$. We have $\exp(\KazT{\varnothing}(0))=1$, cf. Lemma \ref{lem:master1}, thus the Theorem is equivalent to
\begin{align} \label{w:M}
    \M(\tvar,\cvar,\svar)=\M(\tvar,\cvar,0)\cdot \N(\tvar,\svar) \in \QQ[[\tvar,\cvar,\svar]]\,,
\end{align}
For a map $f:M\to N$ in $\CC_l$ we have
\begin{align}\label{w:M2}
	\M(f)=\sum_{\multi\in X_S}
	\frac{\ssm(\Slocus{\multi}{f})}{f^*\ssm(\Tlocus{\varnothing}{f})}\cdot t^\multi
	\in \coh^\bullet(M)[[\tvar]]\,,
\end{align}

\begin{lemma} \label{lem:Mmult}
	Let $f:M_1\to N$ and $g:M_2\to N$ be maps in $\CC_l$. Let $\Gamma$ be the graph of $f$. Then
	$$\Gamma^*\M(f\sq g)_{|M_1\times N}=\M(f)\cdot f^*\N(g) \in
	\coh^\bullet(M_1)[[\tvar]]\,.$$
\end{lemma}
\begin{proof}
	The result follows from Formula \eqref{w:M2} and Corollary \ref{cor:sum}.
\end{proof}

\begin{lemma} \label{lem:Sind}
	We have
	$$\M(\tvar,\cvar,2\cdot\svar)=\M(\tvar,\cvar,\svar)\cdot \N(\tvar,\svar) \in \QQ[[\cvar,\svar,\tvar]]\,,$$
\end{lemma}
\begin{proof}
	We proceed analogously as in Lemmas \ref{lem:MSlin} and \ref{lem:Nlin}.
	Let $f:M\to N$ be a map in $\CC_l$ and $\Gamma:M\to M\times N$ its graph. We have
	\begin{align*}
		(\M(\tvar,\cvar,2\cdot\svar)) (f)=\Gamma^*\M(f^{(2)})_{|M\times N}=\M(f)\cdot f^*\N(f)
		\in\coh^\bullet(M)[[\tvar]] \,.
	\end{align*}
	The first equation follows from Proposition \ref{pro:sum2} and the second from Lemma \ref{lem:Mmult}.
	The above formula holds for every $f\in \CC_l$. Proposition \ref{pro:unique} implies that it holds also in $\QQ[[\cvar,\svar,\tvar]]$.
\end{proof}
\begin{lemma} \label{lem:Sind2}
	For an arbitrary integer $k$ we have
	$$\M(\tvar,\cvar,k\cdot\svar)=\M(\tvar,\cvar,\svar)\cdot \N(\tvar,(k-1)\cdot\svar)\,.$$
\end{lemma}
\begin{proof}
	Lemma \ref{lem:Sind} implies that the desired formula holds for $k=2$. By induction (using Corollary  \ref{cor:Nlin}) we obtain that it holds for $k$ equal to an arbitrary power of $2$. \\
	Choose a monomial in $\QQ[[\cvar,\svar,\tvar]]$. Coefficients corresponding to this monomial in the right and the left hand side are polynomials in $k$. These two polynomials agree for infinitely many values, so they are equal.
\end{proof}
\begin{proof}[Proof of Theorem \ref{thm:indS}]
	Lemma \ref{lem:Sind2} for $k=0$ implies that
	$$\M(\tvar,\cvar,0)=\M(\tvar,\cvar,\svar)\cdot \N(\tvar,(-1)\cdot\svar)\,,$$
	By Corollary \ref{cor:Nlin} for $k=(-1)$ we have
	$$\N(\tvar,(-1)\cdot\svar)=\N(\tvar,\svar)^{-1}\,. $$
	Combining these two formulas we obtain an equivalent form of the theorem from formula \eqref{w:M}.
\end{proof}

\section{Relation between source and target polynomials} \label{s:F}
In this section we study the relation between the source and target SSM-Thom polynomials. \\
Let $(\eta,\multi)$ be an S-multisingularity and $f$ a map in $\CC_l$. Denote by $k$ the multiplicity of the chosen singularity $\eta$ in $\multi$. The restriction $f|_{\Slocus{\multi}{f}}:\Slocus{\multi}{f}\to \Tlocus{\multi}{f}$ is a $k$ to $1$ map. Characteristic classes of the strata $\Slocus{\multi}{f}$ and $\Tlocus{\multi}{f}$ are related by the formulas
\begin{align} \label{w:F}
f_*\left(\frac{\ssm(\Slocus{\multi}{f})}{c_\bullet(T_f)}\right)=
f_*\left(\frac{\csm(\Slocus{\multi}{f})}{f^*c_\bullet(TN)}\right)
=k\cdot \ssm(\Tlocus{\multi}{f})
\,, \qquad 
f_*[\overline{\Slocus{\multi}{f}}]=k\cdot[\overline{\Tlocus{\multi}{f}}]
\,.
\end{align}
These operations can be described on the level of formal power series.
\begin{df}
	Let $\F,\FF:\QQ[[\cvar,\svar]] \to \QQ[[\svar]]$ be maps of $\QQ[[\svar]]-$modules defined by
	$$\F\left(\sum_{\lambda} a_\lambda c_\lambda\right)=\sum_{\lambda} a_\lambda s_\lambda\,,\qquad
    \FF(W)=\F\left(\frac{W}{1+c_1+c_2+\dots}\right)\,.$$
\end{df}

\begin{ex}
	We have
	$$\F(1+s_1+c_1s_1+2c_1^2+3c_2c^2_3+c_2s_1-s_2c_1)=s_\varnothing+s_1s_\varnothing+s_1^2+2s_{1,1}+3s_{2,3,3}\,.$$
\end{ex}
Let $f:M\to N$ be an arbitrary proper map of smooth varieties. The map $F$ describes the pushforward map $f_*:\coh^\bullet(M)\to \coh^{\bullet}(N)$, i.e. for a power series $A \in \QQ[[\cvar,\svar]]$ we have
\begin{align} \label{w:F2}
	\F(A)(f)=f_*(A(f))\in \coh^\bullet(N)\,.
\end{align}
For the classical Thom polynomials we have $\F(\ThScl{\multi})= \ThTcl{\multi}\,. $
The map $\FF$ plays the same role in the case of SSM-Thom polynomials.
\begin{pro} \label{pro:F}
	Assume that Conjecture \ref{con:SSM}[S] holds, then Conjecture \ref{con:SSM}[T] also holds. Moreover, for an S-multisingularity $(\eta,\multi)$ we have
	$$\FF(\ThS{\multi})= \ThT{\multi}\,,\qquad \FF(\KazS{\multi})=\KazT{\multi}\,.$$
\end{pro}
The first statement is present in \cite{nekarda1}. While $\FF$ is not an isomorphism, its restriction to $\QQ[[\cvar]]$ is a $\QQ$-vector space isomorphism onto the subspace of linear power series in variables $\svar$. Hence we obtain 

\begin{cor}\label{cor:RvsS}
Proposition~\ref{pro:F} determines $\KazS{(\eta,\multi)}$ from $\KazT{\multi}$. The series $\KazS{(\eta,\multi)}$ is independent of the distinguished element $\eta$.
\end{cor}

\begin{proof}[Proof of Proposition \ref{pro:F}]
	Let $(\eta,\multi)$ be an S-multisingularity and $f$ a map in $\CC_l$.
	Formulas \eqref{w:F} and \eqref{w:F2} imply that 
	$$
	\FF(\ThS{\multi})(f)=
	f_*\left(\frac{|\Aut(\eta_1,\multi)|\cdot \ssm(\Slocus{\multi}{f})}{c_\bullet(T_f)}\right)=
	|\Aut(\multi)|\cdot \ssm(\Tlocus{\multi}{f})\,.
	$$
	Therefore, the series $\ThT{\multi}:=\FF(\ThS{\multi})$ satisfies  Conjecture \ref{con:SSM}[T].  \\
	For the second part consider the $\svar$-grading on the rings $\QQ[[\cvar,\svar]]$ and $\QQ[[\svar]]$. The operation $W \to \frac{W}{1+c_1+c_2+\dots}$ preserves this grading and the operation $\F$ increases it by one. The polynomial $\KazS{\multi}$ is the degree zero part of $\ThS{\multi}$. Therefore its image in $\FF$ is the $\svar$-linear part of $\FF(\ThS{\multi})$. Thus $\FF(\KazS{\multi})=\KazT{\multi}$.
\end{proof}

\begin{ex}
The independence of $\KazS{(\eta,\multi)}$ on the distinguished element $\eta$ does not imply the same for $\ThS{(\eta,\multi)}$. For example, for $l=1$, up to degree 4 we have
\begin{multline*}
    \ThS{(A_0,\{A_0,A_1\})} = 
    \exp(\KazT{\varnothing}) \left( \KazS{A_0A_1} + \KazS{A_0} \KazT{A_1} \right)  =
    \\
    (-2c_{21}-2c_3+s_2) 
    +(-2c_{211} -5c_{31} -c_{22} -4c_4 +2c_{21}s_0 +c_1s_2 + 2c_3s_0 -s_{2}s_0 +s_3)+\ldots ,
\end{multline*}
\begin{multline*}
    \ThS{(A_1,\{A_0,A_1\})} = 
    \exp(\KazT{\varnothing}) \left( \KazS{A_0A_1} + \KazS{A_1} \KazT{A_0} \right)  =
    \\
    (-2c_{21} - 2c_3 + c_2s_0)+ (-2c_{211} - 5c_{31} - c_{22} - 4c_4 + 3c_{21}s_0 - c_2s_0^2   + 3c_3s_0 )+\ldots .
\end{multline*}
While these are not equal, it is instructive to verify that their $\FF$ images are equal (in particular, the $\F$-images in the lowest degree), in accordance with Proposition~\ref{pro:F}.
\end{ex}

\section{Interpolation} \label{s:interpolation}
In this section, we reformulate Conjecture~\ref{con:SSM2} into a computationally testable statement, one that can be verified algorithmically. This approach, together with explicit computer calculations, yields the desired SSM-Thom polynomials of multisingularities, computed up to a prescribed degree at most the Mather bound.

\subsection{Prototypes of singularities with maximal symmetry}
Let $\eta$ be a singularity, $p_\eta$ its prototype, and let $G_\eta$ denote its {\em maximal compact symmetry group}. We write $\TT_\eta$ for the maximal torus of $G_\eta$, and denote by $\rhoS$ and $\rhoT$ the corresponding representations on the source and target spaces. The map
$f=B_{\TT_{\eta}}p_{\eta}$ 
introduced in \eqref{eq:BGf} in Section~\ref{sec:CC_l}, will play a central role in what follows.

One motivation for focusing on such maps is that many of the ingredients that appear in Thom polynomial formulas take particularly simple forms for such a map $f$. Indeed, we see that $f^*$ is the identity map of $\coh^*(B\TT_\eta)$, and the push-forward map $f_*$ is multiplication by $\eu(\rhoT)/\eu(\rhoS)$, and $c_\bullet(f)=c(\rhoT)/c(\rhoS)$.

\begin{ex} \label{ex:I22_classes}
We met the prototype for $Q=I_{22}$ and $l=1$ in Example \ref{ex:I22,l=1,prototype}, 
\begin{equation*}
p:(x,y,u_1,u_2,u_3,u_4,u_5)\mapsto (x^2+u_1y,y^2+u_2x,u_3x+u_4y+u_5xy,u_1,u_2,u_3,u_4,u_5),
\end{equation*}
whose maximal torus symmetry is $\TT=U(1)^3$ with the representations
\[
\begin{array}{lcll}
   \rhoS & = &  \alpha + \beta & + \alpha^2\bar{\beta} + \beta^2\bar{\alpha} + \gamma\bar{\alpha} + \gamma\bar{\beta} + \gamma\bar{\alpha}\bar{\beta}, \\
   \rhoT & = &  \alpha^2 + \beta^2 + \gamma & + \alpha^2\bar{\beta}+\beta^2\bar{\alpha} + \gamma\bar{\alpha} + \gamma\bar{\beta} + \gamma\bar{\alpha}\bar{\beta},
\end{array}
\]
on the source $\C^7$ and target $\C^8$ spaces. Thus, we have
\[
c_\bullet(f)=\frac{(1+2a)(1+2b)(1+c)}{(1+a)(1+b)}=
1+(a+b+c) + (-a^2-b^2+ab+ac+bc) + \ldots \in \QQ[[a,b,c]].
\]
Moreover, $f^*=\id:\QQ[a,b,c] \to \QQ[a,b,c]$, and for $f_*:\QQ[a,b,c]\to \QQ[a,b,c]$ have
\[ 
f_*(x)= \frac{\eu(\rhoT)}{\eu(\rhoS)} x = \frac{(2a)(2b)(c)}{ab} x = 4c \cdot x,
\]
as well as, for example 
\[
s_{211}(f)=
f_*(c_2(f)c_1(f)^2)=
4c (-a^2-b^2+ab+ac+bc)(a+b+c)^2.
\]
\end{ex}

Our interpolation theorem in the next section reduces the task of finding SSM-Thom polynomials to explicit calculations we just illustrated. 

\subsection{Interpolation theorem}
Let $\multi_0$ be a Mather T-multisingularity and $ k \le M(l)+l$ a degree bound.

Assume that for every $\multi\subset \multi_0$ we are given a linear polynomial $\KazT{\multi}$ in the variables $\svar$ of cohomological degree at most $k$. Define the associated $A_\multi$ polynomials by 
\begin{align*} 
	\sum_{\multi\in X_T}
	\frac{A_{\multi}}{|\Aut(\multi)|}\cdot t^\multi=
	\exp \bigg(
	\sum_{\multi\subset\multi_0}
	\frac{\KazT{\multi}}{|\Aut(\multi)|}\cdot t^\multi
	\bigg)\bigg|_{\le k},
\end{align*} 
where $X$ be the set of singularities occurring in $\multi_0$ and we used the notation $X_T$ from Corollary \ref{cor:indT}. In particular, recall that $t^{\multi}$ are monomials in formal variables associated to monosingularities, and $t^\varnothing=1$. 

\begin{thm}[Interpolation Theorem] \label{thm:Interpolation}
We have 
\begin{align}\tag{$\star$} \label{w:Interpolation}
	A_{\multi}(f)=\ssm(\Tlocus{\multi}{f})_{|\le k} \cdot |\Aut({\multi})| 
\end{align}
for every $\multi\subset \multi_0$ and every $f\in \CC_l$ in the cohomology of the target of $f$, if and only if, the following conditions hold.
\begin{enumerate}
        \item \label{i1} For every monosingularity $\eta\in \multi_0$, polynomial $A_{\{\eta\}}$ satisfies Condition \eqref{w:Interpolation}  for the prototype $p_\eta$ in $\TT_\eta$-equivariant cohomology.
        \item \label{i2} 
        For any $\multi \subset \multi_0$ and monosingularity $\zeta$ with prototype $p_\zeta:M_\zeta\to N_\zeta$ such that $\tcodim(\zeta)\le k$ and $\multi\neq \{\zeta\}$ we have
		$$
        \left(
        A_{\multi}(p_\zeta)\cdot c_\bullet(TN_\zeta)
        \right)|_{r}
        =0 \in \coh_{\TT_\zeta}^{r}(\pt)\,,$$
		for $r\in\{\tcodim(\zeta) ,\dots,k\}.$
        	\end{enumerate}
\end{thm}

The simplest special case of the theorem, $\multi_0=\varnothing$, is already powerful:
\begin{ex}[Calculation of the Master Series] \label{ex:MS}
Let $\multi_0=\varnothing$, $k\le M(l)+l$, and let $\KazT{\varnothing}$ be a linear polynomial of cohomological degree at most $k$. According to Theorem~\ref{thm:Interpolation}, 
    $\ThTMa{\varnothing}{k}=\exp (\KazT{\varnothing})|_{\leq k}$ satisfies Conjecture~\ref{con:SSM2}[T] if and only if for every monosingularity $\zeta$ with $\tcodim(\zeta)\le k$ we have
    \begin{equation}\label{eq:something}
    \left( 
    \exp(\KazT{\varnothing})(p_\zeta)
    \cdot c_\bullet(TN_\zeta)
    \right)
    |_{r}=0 \in \coh_{\TT_\zeta}^{r}(\pt)
    \qquad
    \text{ for }
    r= \tcodim(\zeta),\dots,k.
    \end{equation}
    This statement may sound counter-intuitive, given that all constraints we put on $\exp(\KazT{\varnothing})$ are homogeneous, without a single normalization condition. However, the normalization condition is hidden in the exponential form $\ThT{\varnothing}=\exp (\KazT{\varnothing})$ that forces the degree zero part of $\ThT{\varnothing}$ to be $1$. The verification of \eqref{eq:something} is a matter of explicit polynomial algebra, illustrated in Section~\ref{s:interpolation}, and for small values of $l$ we obtain
    \begin{align*}
\KazT{\varnothing}^{l=1} = &
-s_\varnothing+\tfrac{1}{2} s_1 +\left( \tfrac{7}{6} s_2- \tfrac{1}{3} s_{11} \right) 
+ \left( s_3 - \tfrac{5}{4} s_{21} + \tfrac{1}{4}s_{111}  \right)
+ \dots
\\ 
\KazT{\varnothing}^{l=2} =&
 -s_\varnothing 
+ s_1 
+\left( \tfrac{1}{2}s_2 -s_{11} \right)
+\left( 2s_3-s_{21} +s_{111}\right)
+ \ldots
\\
\KazT{\varnothing}^{l=3} = &
-s_\varnothing
+s_1
+(s_2-s_{11})
+\left( \tfrac{1}{2}s_{3} -2s_{21}+ s_{111} \right)
+\ldots
\\
\KazT{\varnothing}^{l=4} = &
-s_\varnothing
+ s_1
+( s_2-s_{11} )
+( s_3-2s_{21} +s_{111}  )
+ \ldots,
\end{align*}
\[
\KazT{\varnothing}^{l=\infty} = 
\sum_{\lambda} (-1)^{\ell(\lambda)+1} \binom{\ell(\lambda)}{a_1\ a_2\ \ldots\ a_r}s_\lambda
\quad
\text{where }
\quad
\lambda=(1^{a_1}2^{a_2}\ldots),\  \ell(\lambda)=\sum a_i.
\]
Higher degree terms, and $\KazT{\varnothing}$ for other $l$ values, are available on the \cite{TPP}.
\end{ex}
\begin{rem}
    A third -- obvious -- condition can be added to the constraints \eqref{i1}, \eqref{i2} in Theorem~\ref{thm:Interpolation}, namely, $\KazT{\multi}|_{r}=0$ for $r<\tcodim({\multi})$. This is not listed in the theorem, because it is forced by the two listed constraints. Yet, computer calculations can be sped up by adding this third condition.
\end{rem}

\begin{rem}
    An interpolation theorem for $\KazS{\multi}$, analogous to Theorem~\ref{thm:Interpolation}, could be phrased and proved similarly. It would then be an effective algorithm to calculate the $\KazS{\multi}$ series. Alternatively, once $\KazT{\multi}$ is known for a non-empty $\multi$ (up to a degree), and the existence of $\KazS{\multi}$ (up to a degree) is established similarly to that of $\KazT{\multi}$, then the corresponding $R_\multi$ value follows from Corollary~\ref{cor:RvsS}. 
\end{rem}

\subsection{Proof of the Interpolation Theorem} \label{sec:Proof of IT}
First we show the easy direction:
	Suppose that the polynomial $A_{\multi}$ satisfies condition \eqref{w:Interpolation} for every $f\in\CC_l$. Then condition \eqref{i1} is trivial. For the second condition notice that  
    \[
        \ssm\left(\Tlocus{\multi}{\zeta} \subset N_\zeta\right) \cdot c(TN_\zeta) = 
        \csm\left(\Tlocus{\multi}{\zeta} \subset N_\zeta\right),
    \]
    and that $\tcodim(\zeta)=\dim(N_\zeta)$. Then \eqref{i2} follows from Weber's Theorem~\ref{thm:weber}. 

\smallskip

Let us now focus on the other direction of the theorem. First, it is enough to check condition \eqref{w:Interpolation} for maps with controlled singularities.
\begin{lemma}\label{lem:Interp1}
	Let $\multi$ be a Mather T-multisingularity and $B_\multi \in \QQ[\svar]$ a polynomial of degree $k$. If the polynomial $B_{\multi}$ satisfies condition \eqref{w:Interpolation} for every $f\in\CC^{\Ma,k}_l$ then it satisfies it for any $f\in\CC_l$.
\end{lemma}
\begin{proof}
    Let $f:M\to N$ be a map in $\CC_l$ and $D$ denote the difference $A_{\multi_0}(f)-\ssm(\Tlocus{\multi_0}{f})_{|\le k} \cdot |\Aut({\multi_0})|$. Note that $D=D|_{\le k}$. Consider the open subset
    $$U_k=\bigcup_{\tcodim(\mzeta) \le k} \Tlocus{\mzeta}{f}\subset N\,,$$
    and its complement $F_t$. The restricted map map $f|_{f^{-1}(U_k)}$ is in $\CC^{\Ma,k}_l$, therefore $D|_{U_k}=0$. The short exact sequence
    $$\coh_\bullet(F_k)\xto{i_*}\coh^\bullet(N)\to\coh^\bullet(U_k)\to 0$$
    shows that $D$ is in the image of $i_*$. Codimension of the set $F_k$ is at least $k+1$, cf. Section \ref{sec:Mather} , therefore $D|_{\le k}=0$.
\end{proof}
We can restrict the considered class of maps further to prototypes of multisingularities.
\begin{pro} \label{pro:Interp2}
	Let $\multi$ be a Mather T-multisingularity and $B_\multi \in \QQ[\svar]$ a polynomial of degree $k$. Choose $t\le k$. The following conditions are equivalent
	\begin{enumerate}
		\item The polynomial $B_\multi$ satisfies \eqref{w:Interpolation} for any ${f\in \CC^{\Ma,t}_l}$.
		\item The polynomial $B_\multi$ satisfies \eqref{w:Interpolation} for the prototypes of all nonempty multisingularities $\mzeta$, such that $\tcodim(\mzeta)\le t$ in the $\TT_\mzeta$-equivariant cohomology.
	\end{enumerate}
\end{pro}
\begin{proof}
	The proof of this statement is a standard argument in global singularity theory, spelled out in detail (in slightly different circumstances) in \cite[Sect.~6]{rrtp}, \cite[Sect.~3]{LFRRobstructions}. The argument depends on a construction called the {\em universal singular map}, pioneered by A. Sz\H ucs \cite{szucs1, szucs2, szucs3, szucs4, szucs5, szucs6}; we will use the version in \cite{RRSzA}. Since the universal singular map is quite a sophisticated object, we will only sketch it, and refer the reader to the listed references.

Let $\tau$ be the finite set of (necessarily Mather) T-multisingularities with $\tcodim \leq t$. 
For each $\multi \in \tau$ we choose a prototype $p_\multi$, and let its maximal compact symmetry group be $G_\multi$ with source and target representations $\rhoS$ and $\rhoT$. For a monosingularity these are ordinary representations, but for a proper multisingularity $\rhoS$ is acting on $\C^b \sqcup \C^b \sqcup \ldots \sqcup \C^b$ which may include permutations of the components. Consider the Borel construction 
\[
B_{G_\multi} p_\multi : B_{G_\multi}(\C^b \sqcup \C^b \sqcup \ldots \sqcup \C^b) \to B_{G_\multi} \C^{b+l},
\]
(cf. Section \ref{sec:CC_l}) and call it the {\em block corresponding to $\multi$}. The universal singular map $F_\tau: X_\tau \to Y_\tau$ is `glued together' in a particular way from the blocks corresponding to all $\multi \in \tau$. It has the following properties:
\begin{itemize}
\item Every stable map whose multisingularities belong to $\tau$ arises (essentially) uniquely as a pullback from $F_\tau$. The functoriality of this pullback implies that any identity involving characteristic classes of singularity loci and characteristic classes of the map holds for $F_\tau$ if and only if it holds for all stable maps with $T$-singularities in $\tau$.
\item The gluing of the blocks satisfies the so-called Euler condition. Through a combination of a Gysin sequence and a Mayer–Vietoris argument, this implies that a cohomology class in $X_\tau$ (respectively $Y_\tau$) vanishes if and only if it vanishes in the source (respectively target) of each block, for all $\multi \in \tau$.
\end{itemize}
Condition \eqref{w:Interpolation} for $F_\tau$ restricted to a $\multi$-block is the same as \eqref{w:Interpolation} applied to $p_\multi$ in $G_\multi$-equivariant cohomology. Since $\coh^*(BG_\multi)$ is a subring of $\coh^*(B\TT_{\multi})$ the proposition is proved.
\end{proof}

\begin{rem}
The Euler condition mentioned in the proof follows from the experimental fact asserting that in the Mather range every singularity (and hence also every multisingularity) admits a quasihomogeneous prototype \cite[Thm.~7.6]{MNB}.
\end{rem}

The way the $A_{\multi}$ polynomials are constructed from the $\KazT{\multi}$ polynomials implies that we can restrict further to prototypes of monosingularities.
\begin{lemma} \label{lem:unique3}
    Let and $f$ and  $g$ be maps in $\CC_l$. Suppose that for every $\multi\subset\multi_0$ the polynomial $A_{\multi}$ satisfy \eqref{w:Interpolation} for $f$ and $g$.
	Then for every $\multi\subset \multi_0$ the polynomial $A_{\multi}$ satisfy \eqref{w:Interpolation} for $f\sq g$.
\end{lemma} 
\begin{proof}
    Let $\svar_1$ and $\svar_2$ be two sets of varibles. The series $\KazT{\multi}$ are linear, therefore
    $$\sum_{\multi\in X_T}
	\frac{A_{\multi}(\svar_1+\svar_2)t^\multi}{|\Aut(\multi)|} =
	\exp \bigg(
	\sum_{\multi\subset\multi_0}
	\frac{\KazT{\multi}(\svar_1)+\KazT{\multi}(\svar_2)}{|\Aut(\multi)|}\cdot t^\multi
	\bigg)\bigg|_{\le k}
    =
    \sum_{\multi\in X_T}
	\frac{A_{\multi}(\svar_1)t^\multi}{|\Aut(\multi)|}
    \cdot
    \sum_{\multi\in X_T}
	\frac{A_{\multi}(\svar_2)t^\multi}{|\Aut(\multi)|}$$
    Computing coefficient of $t^\multi$ for $\multi\subset\multi_0$ we get
    $$A_{\multi}(\svar_1+\svar_2)=\sum_{\multi_1+\multi_2=\multi}A_{\multi_1}(\svar_1)\cdot A_{\multi_2}(\svar_2) \cdot\frac{|\Aut(\multi)|}{|\Aut(\multi_1)|\cdot|\Aut(\multi_2)|}$$
    By Propositions \ref{pro:sum1} and \ref{pro:sum2} this finishes the proof of the lemma.
\end{proof}
\begin{cor}
    The proof of Lemma \ref{lem:unique3} works also in the equivariant setting. Let $\TT_f$ and $\TT_g$ be algebraic tori, such that the map $f$ is $\TT_f$-equivariant and $g$ is $\TT_g$-equivariant. Suppose that $A_{\multi}$ satisfy condition \eqref{w:Interpolation} for $f$ in $\coh^\bullet_{\TT_f}$ and for $g$ in $\coh^\bullet_{\TT_f}$. Then it satisfies condition \eqref{w:Interpolation} for $f\sq g$ in $\TT_f\times \TT_g$-equivariant cohomology.
\end{cor}
The above considerations may be summarized in the following result.
\begin{cor} \label{cor:Interp}
	Choose $t\le k$.
	Suppose that for every $\multi\subset\multi_0$ and every nonempty singularity $\zeta$ such that $\tcodim(\zeta)\le t$ the polynomial $A_{\multi}$ satisfies condition \eqref{w:Interpolation} for the prototype $p_\zeta$ in the equivariant cohomology. Then it satisfies condition \eqref{w:Interpolation} for any $f\in\CC^{\Ma,t}_l$.
\end{cor}

Now we are ready for the final steps in the proof of Theorem \ref{thm:Interpolation}.

	Suppose that the polynomials $A_{\multi}$ satisfy conditions \eqref{i1} and \eqref{i2}. We will inductively prove that they satisfy condition \eqref{w:Interpolation} for any map in $\CC_l^{\Ma,t}$. For $t=k$ this will prove a result due to Lemma \ref{lem:Interp1}\\
	{\bfseries The case $t=l$:} By Corollary \ref{cor:Interp} we have to show that for every $\multi\subset \multi_0$ the polynomial $A_{\multi}$ satisfies condition \eqref{w:Interpolation} for the prototype of $A_0$ singularity. \\
	The prototype of $A_0$ singularity is $i:\pt \to \C^l$. We have $\tcodim (A_0)=l$ and $\TT:=\TT_{A_0}\simeq(\C^*)^l$. Let
	$$D=A_{\multi}(i)-\ssm(\Tlocus{\multi}{i})_{|\le k} \cdot |\Aut(\multi)| \in \coh_{\TT}^\bullet(\C^l)\,.$$
	Our goal is to show that $D=0$. If $\multi=A_0$ then it is true by condition \eqref{i1}. Assume that $\multi\neq A_0$.\\
	All Landweber-Novikov classes are supported on the image of $i$, therefore $A_{\multi}(i)_{|\C^l -0}$ is the constant term of the polynomial $A_{\multi}$. By the inductive construction of $A_{\multi}$ this is $1$ if $\multi=\varnothing$ and zero otherwise. \\
	On the other hand if $\multi\notin\{\varnothing;A_0\}$ then $\Tlocus{\multi}{i}=\varnothing$ and $\ssm(\Tlocus{\multi}{i})=0$. If $\multi=\varnothing$ then
	$$\ssm(\Tlocus{\multi}{i})_{|\C^l -0}=1-\ssm(0\subset \C^l)_{|\C^l -0}=1\,. $$
	In both cases we obtain $D_{|\C^l -0}=0$. The short exact sequence
	$$
	\coh_{\TT}^\bullet(0)\xto{i_*}\coh_{\TT}^\bullet(\C^l)\to\coh_{\TT}^\bullet(\C^l-0) \to 0\,,
	$$
	implies that $D$ lies in the image of $i_*$. Therefore
	$$D|_r=0\in \coh_{\TT}^r(\C^l)$$
	for $r\in\{0,1,\dots,l-1\}$. We have $\multi\neq A_0$, so condition \eqref{i2} implies that it also vanishes for $r\in\{l,\dots, k\}$. Thus $D=0$.
	\\
	{\bfseries Inductive step:} By Corollary \ref{cor:Interp} we have to show that for every $\multi$ the polynomial $A_{\multi}$ satisfies equivariant condition \eqref{w:Interpolation} for the prototypes of monosingularities of target codimension $t$. \\
	Let $\zeta$ be a singularity, such that $\tcodim(\zeta)=t$. If $\multi=\zeta$ then we are done due to condition \eqref{i1}. Suppose that $\multi\neq\zeta$. Let $p:V\to W$ be the prototype of $\zeta$. Let
	$$D=A_{\multi}(p)-\ssm(\Tlocus{\multi}{p})|_{\le k} \cdot |\Aut(\multi)| \in \coh_{\TT_\zeta}^{\le k}(W)\,.$$
	Consider the restricted map $p'=p_{|V-\{0\}}$
	$$p':V-\{0\}\to W-\{0\}\,.$$
	It has only singularities of codimension smaller than $t$. By the inductive assumption, the polynomial $A_{\multi}$ satisfies condition \eqref{w:Interpolation} for $p'$. Thus $D_{|W-\{0\}}=0$. The  short exact sequence
	$$
	\coh_{\TT_\zeta}^\bullet(0)\xto{i_*}\coh_{\TT_\zeta}^\bullet(W)\to\coh_{\TT_\zeta}^\bullet(W-\{0\}) \to 0\,,
	$$
	implies that $D$ lies in the image of $i_*$. Therefore
	$$D|_r=0\in \coh_{\TT_\zeta}^r(W)$$
	for $r\in\{0,1,\dots,\dim W-1\}$. We have $\multi\neq\zeta$ and $t=\dim W$, so condition \eqref{i2} implies that $D|_r$ vanishes also for $r\in\{\dim W,\dots, k\}$. It follows that $D=0$ and the polynomial $A_{\multi}$ satisfies condition \eqref{w:Interpolation} for $p$.


\section{Application to Mond's conjecture}\label{s:Mond}
\subsection{Mond's conjecture} 
Let $f:(\C^m,0)\to (\C^{m+1},0)$ be a map germ in the Mather region. Explicitly, this means that $m\le M(1)=14$. The celebrated Mond conjecture (for a history see \cite[Rem.~8.1]{MNB}) compares two invariants associated with such a germ: image Milnor number and $\A_e$-codimension. 


The $\A_e$-codimension of the germ $f$, denoted $\A_e$-$\codim(f)$, is one of the standard notions of singularity theory. It is defined as a dimension of a certain vector space associated to the germ $f$, see \cite[Def.~3.6 and Cor.~3.2]{MNB}.
A germ is $\A$-finite if it has a finite $\A_e$-codimension. It is stable if and only if its $\A_e$-codimension is zero. 

 Let $f_t$ be a stable perturbation of $f$, see \cite[Sect.~ 8.3]{MNB} for a precise definition. The image Milnor number $\mu_I(f)$ describes the geometry of its image. This image has the homotopy type of a wedge of $m$-spheres \cite[Prop.~8.3]{MNB}. The image Milnor number $\mu_I(f)$ is the number of spheres, alternatively
\begin{align} \label{w:Milnor}
    \mu_I(f):=(-1)^m(\chi(\im f_t)-1)\,.
\end{align}
The Mond conjecture states that for an $\A$-finite germ $f:(\C^m,0)\to (\C^{m+1},0)$ in the Mather range we have
$$
\mu_I(f) \ge \A_e\text{-}\codim(f)\,.
$$
Moreover if $f$ is quasihomogeneous, i.e. it is stabilized with a linear $\C^*$-action with positive weights, then
$$
\mu_I(f) = \A_e\text{-}\codim(f)\,.
$$
\subsection{Formula for the image Milnor number}
Let $f:\C^m\to \C^{m+1}$ be a quasihomogeneous $\A$-finite map in the Mather range. The image Milnor number $\mu_I(f)$ may be computed from the torus weights on the source and target (usually called {\em weights} and {\em degrees}). Explicit formulas were obtained for $m=2$ in \cite{Mond}, $m=3$ in \cite{ohmotoSMTP}, and $m=4,5$ in \cite{PallaresPenafort}. We generalize these formulas for all $m$ up to the theoretical bound $m\le 14$. Our approach follows that of Ohmoto \cite{ohmotoSMTP} 
and depends on our calculation of the Master Series for $l=1$ in Example~ \ref{ex:MS}.

Ohmoto’s method \cite[Thm. 6.5]{ohmotoSMTP} provides a Thom polynomial computing the 
$\ssm$ class of the image. This result was proved only for 
$m\leq 5$ and for maps with Morin singularities. We present the following generalization.
\begin{pro}
	There is a polynomial
	$\ThT{im}\in \QQ[\svar]$
	such that for any map $f:M\to N$ in $\CC_1$ we have
	$$\ThT{im}(f)=\ssm(\im(f))|_{\le 15}\in \coh^{\le 15}(N).$$
\end{pro}
\begin{proof}
	This proposition is equivalent to Conjecture \ref{con:SSM2}[T] for the empty multisingularity $\multi=\varnothing$ and bound $k=15$. If the polynomial $\ThTMa{\varnothing}{15}$ exists, then
	$$\ThT{\im}=1-\ThTMa{\varnothing}{15}=1-\exp(\MSlog)|_{\le 15}\,.$$
	By Theorem $\ref{thm:Interpolation}$ existence of $\ThTMa{\varnothing}{15}$ may be verified by an explicit computation. Computer algebra software shows that this polynomial exists and computes it. The resulting polynomial $\KazT{\varnothing}$ is presented in Section~\ref{ap:examples}.
\end{proof}
	Suppose that the variety $N$ is compact of dimension at most $15$. Then the above theorem allows to compute the Euler characteristic of the image of $f$, cf. Proposition \ref{pro:csm}[5]
	$$\chi(\im(f))=\int_{N}c_\bullet(TN)\cdot\ThT{\im}(f)\,. $$
	In our case, the target variety $N$ is a vector space $\C^{m+1}$. It is not compact, so we cannot directly use the above formula. Fortunately, there is a generalization involving a $\C^*$-action. Suppose that $V$ is a vector space equipped with a linear $\TT=\C^\star$ action, with no zero weights. For an invariant subvariety $X\subset V$ we have 
	$$\chi(X)=\frac{(c^\TT_\bullet(V)\cdot\ssm_\TT(X))|_{\dim V}}{\eu_\TT(V)}\,. $$
	For a quasihomogeneous stable map $f:\C^m\to \C^{m+1}$, with $m\le 14$ we obtain
	$$\chi(\im(f))=\frac{(c_\bullet(\C^{m+1})\cdot (1-\exp(\MSlog)(f)))|_{m+1}}{\eu(\C^{m+1})}\,. $$
	In fact both sides of the above equation are equal to $1$. The situation gets more interesting, when we relax the assumptions and consider a quasihomogeneous $\A$-finite map $f:\C^m\to \C^{m+1}$, where $m\le 14$. It is proved in \cite[Thm. 6.20]{ohmotoSMTP} that  the same formula computes the Euler characteristic of the image of a stable perturbation $f_t$. Combining this result with formula \eqref{w:Milnor} we obtain
	\begin{align} \label{w:Mond}
		(-1)^m\mu_I(f)+1=\chi(\im(f_t))=\frac{(c_\bullet(\C^{m+1})\cdot (1-\exp(\MSlog)))|_{m+1}}{\eu(\C^{m+1})}\,. 
	\end{align}
\subsection{Algorithm}
	We present an algorithm how to pass from the Master Series $S_\varnothing$ to the image Milnor number $\mu_I(f)$, using formula \eqref{w:Mond}. Let $f:\C^m\to \C^{m+1}$ be a quasihomogeneous $\A$-finite map in the Mather region. Denote the weights on the source by $\alpha_1,\dots,\alpha_m$ and on the target by $\beta_1,\dots,\beta_{m+1}$. We use the following standard notation for symmetric polynomials:
	\begin{itemize}
		\item $h_k(\alpha),h_k(\beta)$ denote the complete symmetric polynomials in variables $\alpha$ and $\beta$, respectively;
		\item $e_k(\alpha),e_k(\beta)$ denote the elementary symmetric polynomials in variables $\alpha$ and $\beta$, respectively;
		\item $e_0(\alpha)=e_0(\beta)=1$;
		\item $e_k(\beta)=0$ for $k>m+1$, $e_k(\alpha)=0$ for $k>m$.
	\end{itemize}
	The characteristic classes of $f$ are computed by 
	\begin{align} \label{w:c(f)}
		&1+ c_1(f)t+c_2(f)t^2+\dots=\frac{(1+\beta_1t)\dots(1+\beta_{m+1}t)}{(1+\alpha_1t)\dots(1+\alpha_{m}t)},
	\\
	\label{w:MondSubstitution}
	&c_k:=c_k(f)=\sum_{i=0}^{k} (-1)^{k-i}e_i(\beta)h_{k-i}(\alpha)\,,\qquad 
	s_0:=\eu(f)=e_{m+1}(\beta)/e_m(\alpha)\,.
	\end{align}
	Due to the adjunction formula for pushforward, for a quasihomogeneous map between affine spaces the Landweber-Novikov classes are determined by the variables $s_\varnothing$ and $c_k$:
	\begin{align} \label{w:MondSubstitution3}
		s_\lambda(f)=s_0\cdot c_{\lambda_1} \cdot c_{\lambda_2} \cdot\ldots\cdot c_{\lambda_l(\lambda)}.
	\end{align}
	\begin{df}
		Define the polynomials $K_d \in\QQ[s_0,\cvar]$ for $1\le d\le 15$ to be the graded parts of $1-\exp(\MSlog)_{|\le 15}$ after the substitution \eqref{w:MondSubstitution3}.
	\end{df}
	\begin{ex}
		Here is the calculation of the polynomials $K_d$ for small values of $d$. The Master Series up to degree $3$ is
		$$ \MSlog=-s_\varnothing+\frac{s_1}{2} +\frac{7s_2-2s_{11}}{6} + \dots .$$
		Hence, up to degree $3$, we have
		$$
		1-\exp(\MSlog)= 
        s_\varnothing +
		\frac{1}{2}(-s_1-s_\varnothing^2)+
		\frac{1}{6}(-7s_2+2s_{11}+3s_1s_\varnothing+s_\varnothing^3) +\dots .
		$$
		After the substitution \eqref{w:MondSubstitution3} we obtain
		\begin{align*}
		&1-\exp(\MSlog)_{|\le 15}= s_0 \cdot(1 +
		\frac{1}{2}(-c_1-s_0)+
		\frac{1}{6}(-7c_2+2c_1^2+3c_1s_0+s_0^2) +\dots)\,, \\
		&K_1=s_0\,, \qquad K_2=\frac{s_0}{2}(-c_1-s_0)\,, \qquad K_3=\frac{s_0}{3!}(-7c_2+2c_1^2+3c_1s_0+s_0^2)\,.
		\end{align*}
	    Similarly
        \begin{align*}
			K_4&=\frac{s_0}{4!}\left(-24c_3 + 30 c_1c_2 - 6c_1^3+28c_2s_0-11c_1^2s_0-6c_1s_0^2-s_0^3  \right)\,, \\
            K_5&=\frac{s_0}{5!}\Big(
            -116c_4+116c_1c_3+248c_2^2-156c_1^2c_2+24c_1^4 + (120c_3-220 c_1c_2 +50 c_1^3)s_0 + 
               \\ &\phantom{=\frac{1}{6!}(}  (-70c_2-35c_1^2)s_0^2 +10 c_1s_0^3+s_0^4\Big)
            \,, \\ 
            K_6&=\frac{s_0}{6!}\Big( -720 c_5 + 660 c_1c_4 + 2160 c_2c_3-660 c_1^2c_3 -2280 c_1c_2^2+960 c_1^3c_2 -120 c_1^5 
                \\ &\phantom{=\frac{1}{6!}(}  +(696c_4-1056c_1c_3-1978 c_2^2+1666 c_1^2c_2-274 c_1^4)  s_0  
                \\ &\phantom{=\frac{1}{6!}(}  +(-360c_3+870c_1c_2-225c_1^3)c_0^2+(140c_2-85c_1^2)s_0^3- 15c_1s_0^4-s_0^5\Big)\,.
        \end{align*}
        The (last) polynomial $K_{15}$ has 508 terms, and its largest coefficient is a 16 digit number. All the polynomials $K_1,\ldots,K_{15}$ are available on \cite{TPP}.
    \end{ex}
	Formula \eqref{w:Mond} is equivalent to the following statement.
	\begin{thm}
		Let $f:\C^m\to \C^{m+1}$ be a quasihomogeneous $\A$-finite map, with $m\le 14$. We have
		$$(-1)^m\mu_I(f)+1=\frac{1}{e_{m+1}(\beta)}\sum_{i=1}^{m+1}K_i(s_0,\cvar)\cdot e_{m+1-i}(\beta)\,, $$
		where $s_0$ and $c_k$ are defined by \eqref{w:MondSubstitution}.
	\end{thm}
	For $m\leq 5$, \cite{PallaresPenafort} presents a formula for $\mu_I(f)$, in a slightly different way:
	\begin{thm*}[{\cite[Thm. 2.1]{PallaresPenafort}}]
		Let $f:\C^m\to \C^{m+1}$ be a quasihomogeneous $\A$-finite map, with $m\le 5$. Then
		$$(-1)^m\mu_I(f)+1=\frac{1}{e_{m}(\alpha)}\sum_{i=0}^{m}L_i(s_0,\cvar)\cdot e_{m-k}(\alpha)\,, $$
		where $s_0$ and $c_k$ are defined by the formula \eqref{w:MondSubstitution} and $B_i$ are polynomials of the form:
		\begin{align*}
			L_0&=1\,,\\
			L_1&=\frac{1}{2!}(c_1-s_0)\,, \\
			L_2&=\frac{1}{3!}(s^2_0-c_1^2-c_2)\,, \\
			L_3&=\frac{1}{4!}(-s_0^3-2 s_0^2c_1+ s_0 c_1^2+16 s_0 c_2+2c_1^3-10c_1c_2)\,, \\
			L_4&=\frac{1}{5!}(s_0^4+5s_0^3c_1+5s_0^2c_1^2-50s_0^2c_2-5s_0c_1^3-20s_0c_1c_2+60s_0c_3-6c_1^4+34c_1^2c_2-64c_1c_3+108c_2^2+4c_4)\,, \\
			L_5&=\frac{1}{6!}(-s_0^5-9s_0^4c_1-25s_0^3c_1^2+110s_0^3c_2-15s_0^2c_1^3+270s_0^2c_1c_2-240s_0^2c_3+26s_0c_1^4+16s_0c_1^2c_2+24s_0c_1c_3 
			\\ &\phantom{=\frac{1}{6!}(}
			-1138s_0c_2^2+336s_0c_4+24c_1^5-156c_1^3c_2+276c_1^2c_3+108c_1c_2^2-396c_1c_4+600c_2c_3)\,. \\
		\end{align*}
	\end{thm*}
	The expressions of the last two theorems are only seemingly different. They are related by the map $\FF$ from Section \ref{s:F}. More precisely, we have
	$$
	s_0\cdot(L_0+L_1t+\dots+L_5t^5)=
	(K_1+K_2t^1+\dots+K_6t^5)\cdot(1+c_1t+\dots+c_5t^5) +o(t^6).
	$$
	Using \eqref{w:c(f)} we obtain
	$$\frac{(L_0+\dots+L_5t^5)\cdot(1+e_1(\alpha)t+\dots+e_5(\alpha)t^5)}{e_m(\alpha)}
	=
	\frac{(K_1+\dots+K_6t^5)\cdot(1+e_1(\beta)t+\dots+e_5(\beta)t^5)}{e_{m+1}(\beta)} +o(t^6).
	$$
	For $m\le 5$, a comparison of the coefficient of $t^m$ shows that both expressions give the same result.

\begin{cor}
    The expression we gave for the image Milnor number is a rational expression in the weights $w$ and the degrees $d$. If that expression is not a positive integer for some values of $w$ and $d$, then there exists no finite quasihomogeneous germ with those data.
\end{cor}

\begin{ex} \label{ex:ImageMilnorExample}
    If there exists a quasihomogeneous finite germ $(\C^{10},0) \to (\C^{11},0)$ with weights  and degrees 
    \[ 
    w=(1,1,2,2,3,4,4,5,5,5), \qquad d=(1,2,2,3,4,4,5,5,6,7,10),
    \]
    then its image Milnor number is 34,938,044. There exists no quasihomogeneous finite germ $(\C^{10},0) \to (\C^{11},0)$ with  $w=(1,1,2,2,3,4,4,5,5,5)$, $d=(1,2,2,3,4,4,5,5,6,7,11)$.
\end{ex}

\section{Appendix: Uniqueness} \label{ap:unique}
\subsection{Test maps}
Fix $m\ge 1$ and $l\ge 1$. \\
For a tuple of positive integers $\avar=(a_1,\dots, a_m)$ let $f_\avar:\C^m\to\C^{m+l}$ be a map given by
$$f_\avar(x_1,\dots,x_m)=(x_1^{a_1},\dots, x_m^{a_m},0,\dots,0)\,.$$
Denote by $F_\avar$ its stable unfolding. We will describe the action of the torus under which the maps $f_\avar$ and $F_\avar$ are equivariant. 

Let $\TT_+ = (\C^*)^{m+l}$ be a torus and $\TT = (\C^*)^{m}$ its coordinate subtorus. Denote by $\alpha_1,\dots,\alpha_m,\beta_1,\dots,\beta_l$ the coordinate characters of $\TT_+$. The characters $\alpha$ correspond to the subtorus $\TT$. The full torus $\TT_+$ acts on the domain and codomain of $f_\avar$ diagonally with weights $\alpha_1,\dots,\alpha_m$ and $a_1\alpha_1,\dots,a_m\alpha_m,\beta_1,\dots,\beta_l$, respectively. The action extends naturally to the unfolding space, and we obtain that the $\TT_+$-equivariant Euler class 
$$\eu(F_\avar)=\prod_{i=1}^m a_i\prod_{i=1}^l \beta_i \in \coh^\bullet_{\TT_+}(\pt)$$
is nonzero. It is also nonzero after restriction to a general one parameter subgroup $\sigma:\C^*\to\TT_+$. 
 
\subsection{Chern variables}

\begin{pro} \label{pro:uniqueC1}
	Consider a polynomial $A\in \QQ[\cvar]$. Suppose that for every  $f \in \CC_l$  we have
	$$A(c(f))=0\in \coh^\bullet(X)\,.$$ Then $A=0$ as a polynomial.
\end{pro}

\begin{proof}
	Pick $m\ge\deg A$.
	For a tuple  $\avar=(a_1,\dots, a_m) \in \ZZ_+^m$ we consider the test maps $f_\avar$ and $F_\avar$ with the  $\TT$ action. Let
	\begin{align} \label{w:c}
	    c_\avar:=c_\bullet(T_{F_\avar})=c_\bullet(T_{f_\avar})=\prod\frac{1+a_i\alpha_i} {1+\alpha_i}\in \coh^\bullet_\TT(\pt) \,.
	\end{align}
	We have $F_\avar\in\CC_l$, therefore for any $\avar \in \ZZ_+^m$ we have
	$$A(c_\avar)=0\in \coh^\bullet_\TT(\pt) \simeq\QQ[\alpha_1,\dots,\alpha_m]\,.$$
	Coefficients of the polynomial $A(c_\avar)$ are polynomials in the variables $\avar$ that vanish for every $\avar\in\ZZ_+^m$. Therefore, they vanish for an arbitrary $\avar\in \QQ^m$. In particular for $\avar=(0,\dots,0)$ we obtain
	$$c_\avar=\prod\frac{1}{1+\alpha_i} $$
	and we still have $A(c_\avar)=0$. The proposition follows form the algebraic independence of complete symmetric polynomials. 
\end{proof}

\begin{figure}
\[
\begin{adjustbox}{angle=90}
\begin{tabular}{| C | C | C | C | C | C | C |}
\hline
& \deg 1 & \deg 2 & \deg 3 & \deg 4 & \deg 5 & \deg 6 \\
\hline\hline
& && & & 
   {\color{red}\frac{1}{30}(29s_4-29s_{31}} 
   & {\color{red} \frac{1}{12}(12 s_{5} -11 s_{41} -36 s_{32} }
\\ 
{\color{red}\varnothing}
 & {\color{red} -s_0}  
 & {\color{red}\frac{1}{2}s_1}  
 & {\color{red} \frac{1}{6} (7s_2-2s_{1^2})}
 & {\color{red}\frac{1}{4}( 4s_3-5s_{21} +s_{1^3})} 
 & {\color{red} -62s_{2^2}+39s_{21^2} }
 & {\color{red} + 11s_{31^2} + 38 s_{2^21}} 
 \\ 
 & & & & &{\color{red} -6s_{1^4})} & {\color{red}-16 s_{21^3} + 2 s_{1^5})}
  \\ 
 
 \hline\hline
{A_0}& s_0 & 0 & -2s_2 & -2s_3& s_{31}+5s_{22}& 2s_5+2s_{41}+10s_{32}
\\ 
 & & & & & & 
 \\ 
 \hline
{A_0^2}& & -s_1 & -s_2 & 3s_3+6s_{21}& 7s_4+8s_{31}+4s_{22}&  
-17s_5 -20s_{41} -22s_{32}
\\ 
 & & & & & & -8s_{311}-28s_{221}
 \\ 
 \hline
 {A_1}& & & s_2 & s_3  & -3s_4-2s_{31}-4s_{22} &
 -7s_5 - 4s_{41}-8s_{32}
\\ 
 & & & & & & 
 \\ 
 \hline
 {A_0^3}& & & 2!(s_2+s_{11}) &2!( 3s_3+3s_{21}) & -2!(11s_4+16s_{31}&  
  -2!( 57s_5 + 63 s_{41} 
\\ 
 & & & & & +7s_{22}+12s_{211}) & +42s_{32} +24s_{311} +24s_{221})
 \\ 
 \hline
 {A_1A_0}& & & & -2s_3-2s_{21}& -4s_4-3s_{31}-s_{22}  &  
26s_5+28s_{41}+22s_{32}
\\ 
 & & & & & & +8s_{311}+16s_{221}
 \\ 
 \hline
 {A_0^4}& & & & -3!(2s_3+3s_{21}+s_{111})   & -3!( 12s_4 +14s_{31} &  
3!(46s_{5}+77s_{41}+32s_{32}
\\ 
 & & & & &+4s_{22} +6s_{211})   & +49s_{311} +34s_{221}+20 s_{2111})
 \\ 
 \hline
 {A_2}& & & & &  2s_4 +s_{31} +s_{22}   &   4s_{5}+2s_{41} +2s_{32}
\\ 
 & & & & & & 
 \\ 
 \hline
  {A_1A_0^2}& & & & & 12s_4 +14 s_{31} &  
   48s_{5}+48s_{41}+24s_{32}
\\ 
 & & & & &  +4s_{22} +6s_{211}& +14s_{311}+10s_{221}
 \\ 
 \hline
  {A_0^5}& & & & & 4!( 6 s_4 +9 s_{31}  &  
  4!(60 s_{5}+80 s_{41}+30 s_{32}
\\ 
 & & & & &+ 2s_{22} +6 s_{211} + s_{1111}) & +40s_{311}+20s_{221}+10s_{2111})
 \\ 
 \hline
   {A_2A_0}& & & & &  &  
    -12s_5-12s_{41}-6s_{32}
\\ 
 & & & & & & -3s_{311}-3s_{221}
 \\ 
 \hline
   {A_1^2}& & & & &  &  
  -12s_5-10s_{41}-8s_{32}
\\ 
 & & & & & & -2s_{311}-4s_{221}
 \\ 
 \hline
   {A_1A_0^3}& & & & &  & 
   24(6s_5+8s_{41}+3s_{32}
 \\ 
 & & & & & & +4s_{311}+2s_{221}+s_{2111})
 \\ 
 \hline
   {A_0^6}& & & & &  &  
-5!( 24 s_5 -38 s_{41} +12 s_{32} +25 s_{311} 
\\ 
 & & & & & & +10 s_{221} +10 s_{2111}+ s_{11111})
 \\ 
 \hline
\end{tabular} 
\end{adjustbox}
\]
\vfill \ 
\caption{$S_\multi$ polynomials for $l=1$ up to cohomological degree 6}\label{fig:Sl1}
\end{figure}


\begin{figure}[htbp]
\[
\begin{adjustbox}{angle=90}
\begin{tabular}{| C | C | C | C | C | C | C |}
\hline
& \deg 0 & \deg 1 & \deg 2 & \deg 3 & \deg 4 & \deg 5 \\
\hline\hline
{A_0}& 1  & c_1 & -c_2 & -c_3-2c_{21}& c_4+3c_{22}-c_{31} & 
3c_5+6c_{32}+2c_{41}
\\ 
 & & & & & & +5c_{221}+c_{311}
 \\ 
 \hline
{A_0^2}&  & -c_1 & -c_2-c_{11} & 3c_3+4c_{21}& 7c_4+3c_{22}+10c_{31}+6c_{211} & 
-17c_5-20c_{32}
\\ 
 & & & & & & -14c_{41}-18c_{221}
 \\ 
 \hline
 {A_1}&  &  & c_2 & c_3+c_{21}& -3c_4-3c_{22}-c_{31} &
   -7c_{5}-6c_{32}-7c_{41}
\\ 
 & & & & & & -4c_{221}-2c_{311}
 \\ 
 \hline
 {A_0^3}&  &  & 2!(c_2+c_{11}) & 2!(3c_3+4c_{21}+c_{111})  & -2!(11c_4+6c_{22} & 
 -2!( 57c_5+74c_{41}
\\ 
 & & & & & +13c_{31}+8c_{221} )& +38c_{32}+39c_{311}+12c_{2111})
 \\ 
 \hline
 {A_1A_0} & &  &  & -2c_3-2c_{21} &  -4c_4-c_2^2-5c_{31}-2c_{211}  & 
  26c_5+24c_{41}+20c_{32}
\\ 
 & & & & & & +13c_{221}+5c_{311}
 \\ 
 \hline
 {A_0^4}&  &  &  & -3!(2c_3+3c_{21}+c_{111}) & -3!(12c_4+16c_{31}+4c_{22}&
 3!(46c_5+65c_{41}+30c_{32}
\\ 
 & & & & & +9c_{211}+c_{1111})  & +27c_{221}+35c_{311}+13c_{2111})
 \\ 
 \hline
 {A_2}&  &  &  & &  2c_4 + c_{31} + c_{22}  & 
 4c_5+4c_{41}+2c_{32}+c_{221}+c_{311}
\\ 
 & & & & & & 
 \\ 
 \hline
 {A_1A_0^2}&  &  &  & & 12c_4+14c_{31}+4c_{22}+6c_{211}   & 
  48c_5+60c_{41}+24c_{32}
  \\ 
 & & & & & & +14c_{221}+28c_{311}+6c_{2111} 
 \\ 
 \hline
 {A_0^5}&  &  &  & &  4!(6c_4+2c_{22}+9c_{31} & 
4!( 60c_5 +30c_{32}+86c_{41}+22c_{221}
\\ 
 & & & & &+6c_{211}+c_{1111})  & +49c_{311}+16c_{2111}+c_{11111})
 \\ 
 \hline
 {A_2A_0}&  &  &  & &  & 
-12c_5-12c_{41}-6c_{32}
\\ 
 & & & & & & -3c_{221}-3c_{311}
 \\ 
 \hline
  {A_1^2}&  &  &  & &  & 
  -2(6c_5+5c_{41}+4c_{32}
\\ 
 & & & & & & +2c_{221}+c_{311})
 \\ 
 \hline
\end{tabular} 
\end{adjustbox}
\]
\vfill \ 
\caption{$R_\multi$ polynomials for $l=1$ up to cohomological degree 5. Recall that $\KazS{\multi}$ does not depend on the distinguished element of $\multi$.}\label{fig:Rl1}
\end{figure}

\subsection{Landweber-Novikov variables}
\begin{pro} \label{pro:uniqueS1'}
	Consider a polynomial $B\in \QQ[\svar]$. Suppose that for every $f\in \CC_l$  we have
	$$B(s(f))=0\in \coh^\bullet(Y)\,.$$
    Then $B=0$ as a polynomial.
\end{pro}
 Fix $m>\deg B$. We consider two vector spaces over $\QQ$:
\begin{itemize}
	\item $V_c$ with a dual basis $\{c_1,...,c_m\}$ 
	\item $V_s$ with a dual basis $\{s_\lambda\}$ where partitions $\lambda$ correspond to variables occurring in $B$. We have $|\lambda|\le m$.
\end{itemize}
We treat $B$ as a polynomial function on the space $V_s$.
 Let $S:V_c\to V_s$ be a polynomial map defined by:
 $$s_\lambda(S(x))=c_\lambda(x):=\prod c_{\lambda_i}(x)\,.$$
 To a one parameter subgroup $\sigma:\C^*\to \TT_+$ and a sequence $\avar\in\ZZ_+^m$ we associate a point $x(\avar,\sigma)\in V_c$ such that
 $$c_i(T_{F_{\avar}})= c_i(x(\avar,\sigma))t^i \in \coh^\bullet_{\C^*}(\pt) \simeq \QQ[t]\,.$$
Consider the set
$$X=\{x(\avar,\sigma)\in V_c|\eu_\sigma(F_\avar)\neq0\in\coh^\bullet_{\C^*}(\pt) \}\subset V_c\,.$$
\begin{lemma} \label{lemma:unique1}
	Suppose that a polynomial $A\in \QQ[\cvar]$ of degree at most $m$ vanishes on $X$. Then $A=0$.
\end{lemma}
\begin{proof}
	For $\avar \in \ZZ_+^m$ and a general one parameter subgroup of $\TT_+$ we have
	$$A(c(F_\avar))=0\in \coh_{\C^*}^\bullet(\pt)\,.$$
	Therefore $A(c(F_\avar))=0\in \coh_{\TT_+}^\bullet(\pt)$. The proof of Proposition \ref{pro:uniqueC1} implies that $A=0$.
\end{proof}
\begin{lemma}
	The polynomial $B$ vanishes on the linear span of the image $S(X)$.
\end{lemma}
\begin{proof}
	Choose arbitrary points $x(\avar_1,\sigma_1),\dots, x(\avar_k,\sigma_k) \in S(X)$. Let $e_1,\dots,e_k$ be nonzero numbers such that
	$$\eu(F_{\avar_i})=e_it^l\in \coh^\bullet_{\sigma_i}(\pt)\simeq \QQ[t]\,.$$
	 Pick any $b_1,\dots,b_k\in \NN$ and consider a map
	$$F=F^{(b_1)}_{\avar_1}\sq\dots \sq F^{(b_k)}_{\avar_k}$$
	with the diagonal $\C^*$-action. This map is in $\CC_l$, therefore $B(s(F))=0 \in \coh^\bullet_{\C^*}(\pt)$. Proposition \ref{pro:sum2} implies that $s(F)$ corresponds to a point
	$$\sum_{i=1}^kb_ie_iS(x(\avar_i,\sigma_i)) \in V_s$$
    The set of all such points is Zariski dense in the linear span of points $S(x(\avar_1,\sigma_1)),\dots, S(x(\avar_k,\sigma_k))$.
\end{proof}
\begin{lemma}
	The image $S(X)$ spans the space $V_s$.
\end{lemma}
\begin{proof}
	Suppose otherwise. Then the set $S(X)$ is contained in a codimension one subspace given by some linear equation
	$$\sum a_\lambda s_\lambda(x)=0\,.$$
	Therefore, the polynomial function $A=\sum a_\lambda c_\lambda$ vanishes on the set $X$. It is of degree at most $m$. Lemma \ref{lemma:unique1} implies that $A=0$.
\end{proof}
The lemmas above show that $B$ vanishes on the whole space $V_s$. Therefore, $B=0$.
\subsection{Chern and Landweber-Novikov variables}

\begin{pro} \label{pro:uniqueSC}
	Consider a polynomial $C\in \QQ[\cvar,\svar]$. Suppose that for every  $f\in\CC_l$ we have
	$$C(c(f),f^*s(f))=0\in \coh^\bullet(X)\,.$$ Then $C=0$ as a polynomial.
\end{pro}
Fix $m>\deg C$. For a sequence $\avar\in \ZZ_+^m$ consider the test map $F_\avar$ with the action of the torus $\TT_+$ and let $c_\avar:=c(F_\avar)$. Consider a polynomial $B_\avar(-)=C(c_\avar,-)\in \QQ[\svar]$. We have $\deg B_\avar\le \deg C <m$.
\begin{lemma}
	For any $\avar\in \ZZ_+^m$ we have $B_\avar=0$.
\end{lemma}
\begin{proof}
	Consider a polynomial
	$$B_\avar'(s):=B_\avar(s+s(F_\avar)) \in \QQ[\svar]\,.$$
	It is enough to show that $B_\avar'=0$.
	Let $g$ be a disjoint sum ($\sq$) of test maps and $H=F_\avar\sq g$. Then
	$$0=C(c(H),s(H))_{|1}=C(c(F_\avar),s(F_\avar)+s(g))=
	B_\avar(s(F_\avar)+s(g))=B_\avar'(s(g))\,,
	$$
	where $|_1$ denotes restriction to the component of the domain corresponding to $F_\avar$. The proof of Proposition \ref{pro:uniqueS1'} uses only test maps and their disjoint sums, therefore $B_\avar'=0$.
\end{proof}
We use the isomorphism $\QQ[\cvar,\svar] \simeq \QQ[\cvar][\svar]$ and consider the coefficient of $C$ corresponding to a given monomial in the $\svar$-variables. It is a polynomial in $\cvar$-variables of degree at most $m$, denote it $A\in \QQ[\cvar]$. Due to the above lemma for any $\avar\in \ZZ_+^m$ we have $A(c_\avar)$, where $c_\avar$ is defined in equation \eqref{w:c}. The proof of Proposition \ref{pro:uniqueC1} implies that $A=0$. All coefficients of $C$ vanish, therefore $C=0$.

\section{Appendix: Examples} \label{ap:examples}

\subsection{Sample 
S 
and 
R 
series}
In Figures~\ref{fig:Sl1} and \ref{fig:Rl1} we present some initial terms of $\KazT{\multi}$ and $\KazS{\multi}$ series for $l=1$. Higher degree terms, data for other multisingularities and other $l$'s are available on the \cite{TPP}. Divisibility and positivity (e.g. along the lines of \cite{pragacz:positivity}) observations for terms in various expansions of these series are subject to future study.

\subsection{The \texorpdfstring{$l=1$}{l=1} Master Series}
Here we present the Master Series for $l=1$, up to the theoretical bound, the Mather bound $M(1)=14$:
\begin{multline*}
    S_\varnothing=(-s_0 ) + 
    \left( \frac{1}{2}  s_1  \right) +
    \left( \frac{1}{6} (7s_2-2s_{11}) \right) +
    \left( \frac{1}{4}( 4s_3-5s_{21} +s_{111})  \right) 
    \\
    + \left( \frac{1}{30}(29s_4-29s_{31}-62s_{22}+39s_{211}-6s_{1^4}) \right) +
\end{multline*}

\begin{multline*}
+\frac{1}{12}\left(12 s_{ 5 } -11 s_{ 4 1 } -36 s_{ 3 2 } + 11 s_{ 3 1^2 } + 38 s_{ 2^2 1 } -16 s_{ 2 1^3 } + 2 s_{ 1^5 }
\right)
\\
+\frac{1}{84}\left(86 s_{ 6 } -86 s_{ 5 1 } -289 s_{ 4 2 } + 72 s_{ 4 1^2 } + 41 s_{ 3^2 } 
+ 375 s_{ 3 2 1 } -72 s_{ 3 1^3 } + 381 s_{ 2^3 } -360 s_{ 2^2 1^2 } + 114 s_{ 2 1^4 } -12 s_{ 1^6 }
\right) \\
\end{multline*}

\begin{multline*}
+\frac{1}{24} (24 s_{ 7 } -26 s_{ 6 1 } -132 s_{ 5 2 } + 26 s_{ 5 1^2 } + 36 s_{ 4 3 } + 91 s_{ 4 2 1 } -19 s_{ 4 1^3 } + 13 s_{ 3^2 1 } 
\\
+ 204 s_{ 3 2^2 } -129 s_{ 3 2 1^2 } + 19 s_{ 3 1^4 } -207 s_{ 2^3 1 } + 130 s_{ 2^2 1^3 } -33 s_{ 2 1^5 } + 3 s_{ 1^7 }) 
\end{multline*}

\begin{multline*}
+\frac{1}{90}(87 s_{ 8 } -87 s_{ 7 1 } -517 s_{ 6 2 } + 107 s_{ 6 1^2 } -87 s_{ 5 3 } + 604 s_{ 5 2 1 } -107 s_{ 5 1^3 } -122 s_{ 4^2 } -209 s_{ 4 3 1 } 
\\ 
+ 923 s_{ 4 2^2 } -438 s_{ 4 2 1^2 } + 65 s_{ 4 1^4 } -418 s_{ 3^2 2 } -39 s_{ 3^2 1^2 } -1527 s_{ 3 2^2 1 } + 545 s_{ 3 2 1^3 } -65 s_{ 3 1^5 } 
\\ 
-1022 s_{ 2^4 } + 1255 s_{ 2^3 1^2 } -590 s_{ 2^2 1^4 } + 125 s_{ 2 1^6 } -10 s_{ 1^8 }
)
\end{multline*}

\begin{multline*}
+\frac{1}{20}(20 s_{ 9 } -17 s_{ 8 1 } -40 s_{ 7 2 } + 17 s_{ 7 1^2 } + 140 s_{ 6 3 } + 247 s_{ 6 2 1 } -27 s_{ 6 1^3 } -280 s_{ 5 4 } + 77 s_{ 5 3 1 } + 560 s_{ 5 2^2 } 
\\ 
-164 s_{ 5 2 1^2 } + 27 s_{ 5 1^4 } -118 s_{ 4^2 1 } -300 s_{ 4 3 2 } + 19 s_{ 4 3 1^2 } -193 s_{ 4 2^2 1 } + 103 s_{ 4 2 1^3 } -13 s_{ 4 1^5 }
\\ 
-80 s_{ 3^3 } -62 s_{ 3^2 2 1 } + 4 s_{ 3^2 1^3 } -500 s_{ 3 2^3 } + 477 s_{ 3 2^2 1^2 } -130 s_{ 3 2 1^4 } + 13 s_{ 3 1^6 } + 502 s_{ 2^4 1 } 
\\ 
-410 s_{ 2^3 1^3 } + 154 s_{ 2^2 1^5 } -28 s_{ 2 1^7 } + 2 s_{ 1^9 }
) 
\end{multline*}

\begin{multline*}
+\frac{1}{132}(
142 s_{ 10 } -142 s_{ 9 1 } -571 s_{ 8 2 } + 76 s_{ 8 1^2 } + 3422 s_{ 7 3 } + 713 s_{ 7 2 1 } -76 s_{ 7 1^3 } + 540 s_{ 6 4 } + 2902 s_{ 6 3 1 } 
\\ 
+ 3915 s_{ 6 2^2 } -1647 s_{ 6 2 1^2 } + 208 s_{ 6 1^4 } -2447 s_{ 5^2 } + 262 s_{ 5 4 1 } + 3617 s_{ 5 3 2 } + 322 s_{ 5 3 1^2 } 
\\ 
-4628 s_{ 5 2^2 1 } + 1723 s_{ 5 2 1^3 } -208 s_{ 5 1^5 } + 934 s_{ 4^2 2 } + 161 s_{ 4^2 1^2 } -1387 s_{ 4 3^2 } + 4681 s_{ 4 3 2 1 } 
\\ 
-246 s_{ 4 3 1^3 } -3684 s_{ 4 2^3 } + 2315 s_{ 4 2^2 1^2 } -677 s_{ 4 2 1^4 } + 76 s_{ 4 1^6 } + 463 s_{ 3^3 1 } + 3560 s_{ 3^2 2^2 } 
\\ 
+ 302 s_{ 3^2 2 1^2 } + 19 s_{ 3^2 1^4 } + 8312 s_{ 3 2^3 1 } -4038 s_{ 3 2^2 1^3 } + 885 s_{ 3 2 1^5 } -76 s_{ 3 1^7 } + 4094 s_{ 2^5 } 
\\ 
-6108 s_{ 2^4 1^2 } + 3735 s_{ 2^3 1^4 } -1168 s_{ 2^2 1^6 } + 186 s_{ 2 1^8 } -12 s_{ 1^{10} }
)
\end{multline*}

\begin{multline*}
+\frac{1}{24}(
24 s_{ 11 } -34 s_{ 10, 1 } -636 s_{ 9 2 } + 34 s_{ 9 1^2 } -212 s_{ 8 3 } -759 s_{ 8 2 1 } -1 s_{ 8 1^3 } + 712 s_{ 7 4 } -814 s_{ 7 3 1 } 
\\ 
-796 s_{ 7 2^2 } -559 s_{ 7 2 1^2 } + 1 s_{ 7 1^4 } + 280 s_{ 6 5 } + 940 s_{ 6 4 1 } -1428 s_{ 6 3 2 } -326 s_{ 6 3 1^2 } -2741 s_{ 6 2^2 1 } 
\\ 
+ 336 s_{ 6 2 1^3 } -45 s_{ 6 1^5 } + 841 s_{ 5^2 1 } + 3588 s_{ 5 4 2 } + 298 s_{ 5 4 1^2 } + 264 s_{ 5 3^2 } -1259 s_{ 5 3 2 1 } + 135 s_{ 5 3 1^3 } 
\\ 
-3332 s_{ 5 2^3 } + 1044 s_{ 5 2^2 1^2 } -433 s_{ 5 2 1^4 } + 45 s_{ 5 1^6 } -8 s_{ 4^2 3 } + 1762 s_{ 4^2 2 1 } -37 s_{ 4^2 1^3 } + 713 s_{ 4 3^2 1 } 
\\ 
+ 2368 s_{ 4 3 2^2 } -647 s_{ 4 3 2 1^2 } + 73 s_{ 4 3 1^4 } + 184 s_{ 4 2^3 1 } -446 s_{ 4 2^2 1^3 } + 115 s_{ 4 2 1^5 } -12 s_{ 4 1^7 } 
\\ 
+ 940 s_{ 3^3 2 } -65 s_{ 3^3 1^2 } + 396 s_{ 3^2 2^2 1 } + 16 s_{ 3^2 2 1^3 } -14 s_{ 3^2 1^5 } + 1844 s_{ 3 2^4 } -2380 s_{ 3 2^3 1^2 } 
\\ 
+ 879 s_{ 3 2^2 1^4 } -160 s_{ 3 2 1^6 } + 12 s_{ 3 1^8 } -1846 s_{ 2^5 1 } + 1834 s_{ 2^4 1^3 } 
-896 s_{ 2^3 1^5 } + 240 s_{ 2^2 1^7 } -34 s_{ 2 1^9 } + 2 s_{ 1^{11} }
)
\end{multline*}

\begin{align*}
&
\frac{1}{5460}(
4078 s_{ 12 } -4078 s_{ 11, 1 } -68688 s_{ 10, 2 } + 13178 s_{ 10, 1^2 } -566458 s_{ 9 3 } + 72766 s_{ 9 2 1 } -13178 s_{ 9 1^3 } 
\\ & \hskip 1 true cm
-905251 s_{ 8 4 } -1176391 s_{ 8 3 1 } + 2622 s_{ 8 2^2 } + 31173 s_{ 8 2 1^2 } -4840 s_{ 8 1^4 } + 187022 s_{ 7 5 } 
\\ & \hskip 1 true cm
-1149091 s_{ 7 4 1 } -2280494 s_{ 7 3 2 } -589447 s_{ 7 3 1^2 } -75388 s_{ 7 2^2 1 } -17995 s_{ 7 2 1^3 } + 4840 s_{ 7 1^5 } 
\\ & \hskip 1 true cm
+ 531611 s_{ 6^2 } + 1053876 s_{ 6 5 1 } -458271 s_{ 6 4 2 } -299287 s_{ 6 4 1^2 } -245652 s_{ 6 3^2 } -2240851 s_{ 6 3 2 1 } 
\\ & \hskip 1 true cm
-80655 s_{ 6 3 1^3 } -847752 s_{ 6 2^3 } + 559015 s_{ 6 2^2 1^2 } -140185 s_{ 6 2 1^4 } + 12320 s_{ 6 1^6 } + 1286723 s_{ 5^2 2 } 
\\ & \hskip 1 true cm
+ 157163 s_{ 5^2 1^2 } + 210449 s_{ 5 4 3 } -119784 s_{ 5 4 2 1 } + 50385 s_{ 5 4 1^3 } -183197 s_{ 5 3^2 1 } -1555048 s_{ 5 3 2^2 } 
\\ & \hskip 1 true cm
-280315 s_{ 5 3 2 1^2 } -45545 s_{ 5 3 1^4 } + 923140 s_{ 5 2^3 1 } -541020 s_{ 5 2^2 1^3 } + 135345 s_{ 5 2 1^5 } -12320 s_{ 5 1^7 } 
\\ & \hskip 1 true cm
+ 40823 s_{ 4^3 } + 191242 s_{ 4^2 3 1 } -115564 s_{ 4^2 2^2 } -139930 s_{ 4^2 2 1^2 } + 9760 s_{ 4^2 1^4 } + 595308 s_{ 4 3^2 2 } 
\\ & \hskip 1 true cm
-183675 s_{ 4 3^2 1^2 } -1681260 s_{ 4 3 2^2 1 } + 247470 s_{ 4 3 2 1^3 } -24360 s_{ 4 3 1^5 } + 365270 s_{ 4 2^4 } -225900 s_{ 4 2^3 1^2 } 
\\ & \hskip 1 true cm
+ 90615 s_{ 4 2^2 1^4 } -22540 s_{ 4 2 1^6 } + 2310 s_{ 4 1^8 } -5472 s_{ 3^4 } -213060 s_{ 3^3 2 1 } + 22590 s_{ 3^3 1^3 } 
\\ & \hskip 1 true cm
-715060 s_{ 3^2 2^3 } -84960 s_{ 3^2 2^2 1^2 } -30870 s_{ 3^2 2 1^4 } + 6020 s_{ 3^2 1^6 } -1288410 s_{ 3 2^4 1 } + 766920 s_{ 3 2^3 1^3 } 
\\ & \hskip 1 true cm
-225960 s_{ 3 2^2 1^5 } + 34860 s_{ 3 2 1^7 } -2310 s_{ 3 1^9 } -491460 s_{ 2^6 } + 858690 s_{ 2^5 1^2 } -639240 s_{ 2^4 1^4 } 
\\ & \hskip 1 true cm
+ 259980 s_{ 2^3 1^6 } -60900 s_{ 2^2 1^8 } + 7770 s_{ 2 1^{10} } -420 s_{ 1^{12} }
)
\end{align*}
\begin{align*}
  &  +\frac{1}{840}(
840 s_{ 13 } + 542 s_{ 12, 1 } + 141960 s_{ 11, 2 } -542 s_{ 11, 1^2 } + 235760 s_{ 10 3 } + 353728 s_{ 10 2 1 } -4008 s_{ 10 1^3 } 
\\ & \hskip 1 true cm
-276780 s_{ 9 4 } + 381098 s_{ 9 3 1 } + 437640 s_{ 9 2^2 } + 251614 s_{ 9 2 1^2 } + 4008 s_{ 9 1^4 } -456820 s_{ 8 5 } -537379 s_{ 8 4 1 } 
\\ & \hskip 1 true cm
+ 243180 s_{ 8 3 2 } + 222501 s_{ 8 3 1^2 } + 665178 s_{ 8 2^2 1 } + 84362 s_{ 8 2 1^3 } + 1998 s_{ 8 1^5 } + 229040 s_{ 7 6 } 
\\ & \hskip 1 true cm
-385542 s_{ 7 5 1 } -573580 s_{ 7 4 2 } -299979 s_{ 7 4 1^2 } -135240 s_{ 7 3^2 } + 552894 s_{ 7 3 2 1 } + 16707 s_{ 7 3 1^3 } 
\\ & \hskip 1 true cm
+ 400260 s_{ 7 2^3 } + 319688 s_{ 7 2^2 1^2 } + 29230 s_{ 7 2 1^4 } -1998 s_{ 7 1^6 } + 159779 s_{ 6^2 1 } -118440 s_{ 6 5 2 } 
\\ & \hskip 1 true cm
-101456 s_{ 6 5 1^2 } -178500 s_{ 6 4 3 } -606359 s_{ 6 4 2 1 } -62683 s_{ 6 4 1^3 } -134448 s_{ 6 3^2 1 } + 291060 s_{ 6 3 2^2 } 
\\ & \hskip 1 true cm
+ 255441 s_{ 6 3 2 1^2 } -11885 s_{ 6 3 1^4 } + 755352 s_{ 6 2^3 1 } -84590 s_{ 6 2^2 1^3 } + 30362 s_{ 6 2 1^5 } -2292 s_{ 6 1^7 } 
\\ & \hskip 1 true cm
+ 28560 s_{ 5^2 3 } -380573 s_{ 5^2 2 1 } + 17452 s_{ 5^2 1^3 } + 47740 s_{ 5 4^2 } -136039 s_{ 5 4 3 1 } -898940 s_{ 5 4 2^2 } 
\\ & \hskip 1 true cm
-168476 s_{ 5 4 2 1^2 } -11885 s_{ 5 4 1^4 } -154980 s_{ 5 3^2 2 } -41853 s_{ 5 3^2 1^2 } + 390108 s_{ 5 3 2^2 1 } -64610 s_{ 5 3 2 1^3 } 
\\ & \hskip 1 true cm
+ 9887 s_{ 5 3 1^5 } + 563220 s_{ 5 2^4 } -161540 s_{ 5 2^3 1^2 } + 132640 s_{ 5 2^2 1^4 } -28364 s_{ 5 2 1^6 } + 2292 s_{ 5 1^8 } 
\\ & \hskip 1 true cm
+ 8667 s_{ 4^3 1 } -21840 s_{ 4^2 3 2 } -50682 s_{ 4^2 3 1^2 } -478596 s_{ 4^2 2^2 1 } + 31500 s_{ 4^2 2 1^3 } -1563 s_{ 4^2 1^5 } 
\\ & \hskip 1 true cm
-32200 s_{ 4 3^3 } -352128 s_{ 4 3^2 2 1 } + 26210 s_{ 4 3^2 1^3 } -459340 s_{ 4 3 2^3 } + 257040 s_{ 4 3 2^2 1^2 } -61865 s_{ 4 3 2 1^4 } 
\\ & \hskip 1 true cm
+ 5124 s_{ 4 3 1^6 } + 87770 s_{ 4 2^4 1 } + 21360 s_{ 4 2^3 1^3 } -8848 s_{ 4 2^2 1^5 } + 2628 s_{ 4 2 1^7 } -290 s_{ 4 1^9 } 
\\ & \hskip 1 true cm
-14408 s_{ 3^4 1 } -216860 s_{ 3^3 2^2 } + 16080 s_{ 3^3 2 1^2 } -4775 s_{ 3^3 1^4 } -78040 s_{ 3^2 2^3 1 } -4360 s_{ 3^2 2^2 1^3 } 
\\ & \hskip 1 true cm
+ 10808 s_{ 3^2 2 1^5 } -1416 s_{ 3^2 1^7 } -205800 s_{ 3 2^5 } + 342570 s_{ 3 2^4 1^2 } -154000 s_{ 3 2^3 1^4 } + 37212 s_{ 3 2^2 1^6 } 
\\ & \hskip 1 true cm
-4920 s_{ 3 2 1^8 } + 290 s_{ 3 1^{10} } + 205860 s_{ 2^6 1 } -239680 s_{ 2^5 1^3 } + 142632 s_{ 2^4 1^5 } -49680 s_{ 2^3 1^7 } 
\\ & \hskip 1 true cm
+ 10340 s_{ 2^2 1^9 } -1200 s_{ 2 1^{11} } + 60 s_{ 1^{13} }
)
\end{align*}

\begin{align*}
&
+\frac{1}{360}(
780 s_{ 14 } -780 s_{ 13, 1 } + 6374 s_{ 12, 2 } -1984 s_{ 12, 1^2 } + 453540 s_{ 11, 3 } -5594 s_{ 11, 2, 1 } + 1984 s_{ 11, 1^3 } 
\\ & \hskip 1 true cm
+ 664390 s_{ 10, 4 } + 1000670 s_{ 10, 3, 1 } + 60966 s_{ 10, 2^2 } -27912 s_{ 10, 2 1^2 } + 3476 s_{ 10, 1^4 } -104820 s_{ 9 5 } 
\\ & \hskip 1 true cm
+ 1107590 s_{ 9 4 1 } + 1161886 s_{ 9 3 2 } + 815714 s_{ 9 3 1^2 } -55372 s_{ 9 2^2 1 } + 25928 s_{ 9 2 1^3 } -3476 s_{ 9 1^5 } 
\\ & \hskip 1 true cm
-325254 s_{ 8 6 } -266166 s_{ 8 5 1 } + 1040631 s_{ 8 4 2 } + 613754 s_{ 8 4 1^2 } + 257697 s_{ 8 3^2 } + 1984049 s_{ 8 3 2 1 } 
\\ & \hskip 1 true cm
+ 291462 s_{ 8 3 1^3 } + 69723 s_{ 8 2^3 } -71156 s_{ 8 2^2 1^2 } + 19502 s_{ 8 2 1^4 } -1672 s_{ 8 1^6 } -1830 s_{ 7^2 } 
\\ & \hskip 1 true cm
-465246 s_{ 7 6 1 } -343634 s_{ 7 5 2 } -250970 s_{ 7 5 1^2 } -110650 s_{ 7 4 3 } + 1314093 s_{ 7 4 2 1 } + 89862 s_{ 7 4 1^3 } 
\\ & \hskip 1 true cm
+ 372973 s_{ 7 3^2 1 } + 1590368 s_{ 7 3 2^2 } + 937589 s_{ 7 3 2 1^2 } + 57862 s_{ 7 3 1^4 } -14351 s_{ 7 2^3 1 } + 45228 s_{ 7 2^2 1^3 } 
\\ & \hskip 1 true cm
-16026 s_{ 7 2 1^5 } + 1672 s_{ 7 1^7 } -479753 s_{ 6^2 2 } -171962 s_{ 6^2 1^2 } -156726 s_{ 6 5 3 } -1123134 s_{ 6 5 2 1 } -61260 s_{ 6 5 1^3 } 
\\ & \hskip 1 true cm
-34466 s_{ 6 4^2 } -103252 s_{ 6 4 3 1 } + 314617 s_{ 6 4 2^2 } + 331141 s_{ 6 4 2 1^2 } -3928 s_{ 6 4 1^4 } + 299013 s_{ 6 3^2 2 } 
\\ & \hskip 1 true cm
+ 161151 s_{ 6 3^2 1^2 } + 1648689 s_{ 6 3 2^2 1 } + 105231 s_{ 6 3 2 1^3 } + 7404 s_{ 6 3 1^5 } + 294641 s_{ 6 2^4 } -273423 s_{ 6 2^3 1^2 } 
\\ & \hskip 1 true cm
+ 104010 s_{ 6 2^2 1^4 } -18052 s_{ 6 2 1^6 } + 1188 s_{ 6 1^8 } + 100195 s_{ 5^2 4 } -137329 s_{ 5^2 3 1 } -705746 s_{ 5^2 2^2 } 
\\ & \hskip 1 true cm
-162272 s_{ 5^2 2 1^2 } -11428 s_{ 5^2 1^4 } + 38196 s_{ 5 4^2 1 } -197547 s_{ 5 4 3 2 } -46776 s_{ 5 4 3 1^2 } + 55662 s_{ 5 4 2^2 1 } 
\\ & \hskip 1 true cm
-45891 s_{ 5 4 2 1^3 } + 7404 s_{ 5 4 1^5 } + 5823 s_{ 5 3^3 } + 259245 s_{ 5 3^2 2 1 } + 22047 s_{ 5 3^2 1^3 } + 767029 s_{ 5 3 2^3 } 
\\ & \hskip 1 true cm
+ 225378 s_{ 5 3 2^2 1^2 } + 49245 s_{ 5 3 2 1^4 } -5732 s_{ 5 3 1^6 } -280290 s_{ 5 2^4 1 } + 228195 s_{ 5 2^3 1^3 } -87984 s_{ 5 2^2 1^5 } 
\\ & \hskip 1 true cm
+ 16380 s_{ 5 2 1^7 } -1188 s_{ 5 1^9 } -32265 s_{ 4^3 2 } + 18075 s_{ 4^3 1^2 } -36162 s_{ 4^2 3^2 } -168294 s_{ 4^2 3 2 1 } 
\\ & \hskip 1 true cm
+ 15933 s_{ 4^2 3 1^3 } + 9862 s_{ 4^2 2^3 } + 93519 s_{ 4^2 2^2 1^2 } -15645 s_{ 4^2 2 1^4 } + 592 s_{ 4^2 1^6 } -16839 s_{ 4 3^3 1 } 
\\ & \hskip 1 true cm
-261609 s_{ 4 3^2 2^2 } + 206397 s_{ 4 3^2 2 1^2 } -15180 s_{ 4 3^2 1^4 } + 758865 s_{ 4 3 2^3 1 } -186375 s_{ 4 3 2^2 1^3 } 
\\ & \hskip 1 true cm
+ 39912 s_{ 4 3 2 1^5 } -2856 s_{ 4 3 1^7 } -42207 s_{ 4 2^5 } -2115 s_{ 4 2^4 1^2 } + 7680 s_{ 4 2^3 1^4 } -252 s_{ 4 2^2 1^6 } 
\\ & \hskip 1 true cm
-642 s_{ 4 2 1^8 } + 96 s_{ 4 1^{10} } + 3981 s_{ 3^4 2 } -1383 s_{ 3^4 1^2 } + 100605 s_{ 3^3 2^2 1 } -12105 s_{ 3^3 2 1^3 } 
\\ & \hskip 1 true cm
+ 2592 s_{ 3^3 1^5 } + 211215 s_{ 3^2 2^4 } + 47790 s_{ 3^2 2^3 1^2 } + 16560 s_{ 3^2 2^2 1^4 } -8064 s_{ 3^2 2 1^6 } + 834 s_{ 3^2 1^8 } 
\\ & \hskip 1 true cm
+ 322497 s_{ 3 2^5 1 } 
-226080 s_{ 3 2^4 1^3 } + 80304 s_{ 3 2^3 1^5 } -16128 s_{ 3 2^2 1^7 } + 1830 s_{ 3 2 1^9 } -96 s_{ 3 1^{11} } + 98301 s_{ 2^7 } 
\\ & \hskip 1 true cm
-196512 s_{ 2^6 1^2 } + 171528 s_{ 2^5 1^4 } -85008 s_{ 2^4 1^6 } + 25890 s_{ 2^3 1^8 } -4848 s_{ 2^2 1^{10} } + 516 s_{ 2 1^{12} } -24 s_{ 1^{14} }
)
\end{align*}
 As explained in Section \ref{s:Mond}, this series calculates---in an indirect way---the image Milnor number of quasihomogeneous finite 
 map germs --- just from their weights and degrees. The same information, in various other forms, are available on the \cite{TPP}, in a format that permits copying.

It is remarkable that the denominators of the homogeneous components above coincide with the sequence of Nørlund numbers A002790 \cite{oeis}.

\end{document}